\documentclass[reqno]{amsart}
\usepackage{graphicx}
\usepackage[normalem]{ulem}
\usepackage[foot]{amsaddr}
\usepackage[a-1b]{pdfx}
\usepackage[utf8]{inputenc}
\usepackage[T1]{fontenc}
\usepackage{amsfonts}
\usepackage{amssymb}
\usepackage{amsthm}
\usepackage{mathtools}
\mathtoolsset{mathic=true}
\usepackage{comment}
\usepackage{esint}
\usepackage{hyperref}
\usepackage{xcolor}
\hypersetup{
  final,
   colorlinks,
   linkcolor={red!70!black},
   citecolor={blue!70!black},
   urlcolor={blue!90!black}
}
\usepackage{microtype}
\usepackage{dutchcal}
\usepackage{upgreek}
\usepackage{fourier}

\usepackage{pgfornament}

\newcommand{\HH}{\mathbb{H}}

\newcommand{\R}{\mathbb{R}}

\newcommand{\cH}{\mathcal{H}}

\newcommand{\cL}{\mathcal{L}}

\newcommand{\cN}{\mathcal{N}}

\newcommand{\cS}{\mathcal{S}}

\newcommand{\Kor}{\mathsf{Kor}}

\renewcommand{\epsilon}{\varepsilon}

\newcommand{\ud}[0]{\,\mathrm{d}}

\newcommand{\0}{\mathbf{0}}
\newcommand{\COW}{C^1_{\mathbb{H}}}

\DeclareMathOperator{\II}{II}

\DeclareMathOperator{\rot}{rot}
\DeclareMathOperator{\area}{area}
\DeclareMathOperator{\slope}{slope}

\DeclareMathOperator{\BP}{BP}

\makeatletter
\def\resetMathstrut@{%
  \setbox\z@\hbox{%
    \mathchardef\@tempa\mathcode`\(\relax
    \def\@tempb##1"##2##3{\the\textfont"##3\char"}%
    \expandafter\@tempb\meaning\@tempa \relax
  }%
  \ht\Mathstrutbox@1.2\ht\z@ \dp\Mathstrutbox@1.2\dp\z@
}
\makeatother
\newcommand{\one}{\mathbf{1}}

\newcommand{\from}{\colon}
\newcommand{\symdiff}{\mathbin{\triangle}}
\renewcommand{\mid}{:}

\renewcommand{\ge}{\geqslant}

\renewcommand{\le}{\leqslant}
\renewcommand{\setminus}{\smallsetminus}

\renewcommand{\emptyset}{\varnothing}

\DeclareMathOperator{\id}{id}
\DeclareMathOperator{\Per}{Per}
\DeclareMathOperator{\Lip}{Lip}
\let\div\relax
\DeclareMathOperator{\div}{div}

\newtheorem{thm}{Theorem}[section]

\newtheorem{lemma}[thm]{Lemma}
\newtheorem{prop}[thm]{Proposition}
\newtheorem{cor}[thm]{Corollary}
\theoremstyle{remark}

\newtheorem{definition}[thm]{Definition}
\newtheorem{remark}[thm]{Remark}

\title{Area-minimizing ruled graphs and the Bernstein problem in the Heisenberg group}
\author{Robert Young}
\address{New York University, Courant Institute of Mathematical Sciences, 251 Mercer Street, New York, NY 10012, USA.}
\email{ryoung@cims.nyu.edu}

\begin{document}

\begin{abstract}
  In this paper, we give a necessary and sufficient condition for a graphical strip in the Heisenberg group $\HH$ to be area-minimizing in the slab $\{-1<x<1\}$. We show that our condition is necessary by introducing a family of deformations of graphical strips based on varying a vertical curve. We show that it is sufficient by showing that strips satisfying the condition have monotone epigraphs. We use this condition to show that any area-minimizing ruled entire intrinsic graph in the Heisenberg group is a vertical plane and to find a boundary curve that admits uncountably many fillings by area-minimizing surfaces.
\end{abstract}

\maketitle

\section{Introduction}

Bernstein's theorem states that an entire area-minimizing codimension--$1$ graph in $\R^n$ is a plane when $n\le 8$. Analogues of Bernstein's theorem hold for many classes of surfaces in the Heisenberg group $\HH$. For example, area-minimizing intrinsic graphs of entire $C^2$ functions are vertical planes \cite{BASCV}, as are area-minimizing intrinsic graphs of entire $C^1$ functions \cite{GRBernstein} and locally Lipschitz functions \cite{NSCBernstein}. The problem of determining which classes of graphs satisfy an analogue of Bernstein's theorem is known as the Bernstein problem; for a survey of the Bernstein problem in $\HH$, see \cite{SCVBernstein}. A key difficulty in proving analogues of Bernstein's theorem in $\HH$ is finding appropriate classes of variations of surfaces. In this paper, we introduce new variations of ruled surfaces and use them to solve the Bernstein problem for ruled entire intrinsic graphs.

We briefly recall some terminology. Let $\HH$ be the three-dimensional Heisenberg group; we identify $\HH$ with $\R^3$ with the multiplication
\begin{equation}\label{eq:heisMult}
  (x,y,z)\cdot (x',y',z')=\left(x+x',y+y',z+z'+\frac{xy'-yx'}{2}\right).
\end{equation}
Let $X=(1,0,0)$, $Y=(0,1,0)$, and $Z=(0,0,1)$; since a nilpotent Lie group can be identified with its Lie algebra via the exponential map, we also refer to the corresponding left-invariant vector fields as $X_{(x,y,z)}=(1,0,-\frac{y}{2})$, $Y_{(x,y,z)}=(0,1,\frac{x}{2})$, and $Z_{(x,y,z)}=(0,0,1)$.
These fields generate $1$--parameter subgroups and we denote the elements of these subgroups by $X^t=(t,0,0)$, $Y^t=(0,t,0)$, and $Z^t=(0,0,t)$ for $t\in \R$. Let $V_0=\{(x,0,z)\in \HH\}$ be the $xz$--plane. For a subset $U\subset V_0$ and a function $f\from U\to \R$, we define the intrinsic graph of $f$ to be the set
$$\Gamma_f=\{uY^{f(u)}\mid u\in U\}.$$
If $U=V_0$, we call $f$ an \emph{entire} intrinsic graph. There is a corresponding intrinsic projection $\Pi\from \HH\to V_0$ given by $\Pi(p)=pY^{-y(p)}$, where $y(p)$ is the $y$--coordinate of $p$.

For any $p\in\HH$, we call the plane spanned by $X_p$ and $Y_p$ at $p$ the \emph{horizontal plane} centered at $p$. Let $\mathsf{A}=\{(x,y,0)\mid x,y\in \R\}$ be the horizontal plane centered at $\0$. For any $p\in \HH$ and any $v\in \mathsf{A}$, let $\langle v \rangle$ be the span of $v$; this is a one-parameter subgroup of $\HH$ and we call the coset $p\langle v\rangle$ a \emph{horizontal line}.

One of the main steps in the proofs of the analogues of Bernstein's theorem for $C^2$, $C^1$, and Lipschitz graphs is to calculate the first variation of the area and show that a graph that is area-stationary (i.e., a critical point of the area) is a ruled surface. A \emph{ruled surface} in $\HH$ is a surface $\Sigma\subset \HH$ that can be written as a union of horizontal line segments (\emph{rulings}) with endpoints in $\partial\Sigma$. One then calculates the second variation of the area of $\Sigma$ and shows that if $\Sigma$ is a ruled entire intrinsic graph which is area-stable (i.e., a second-order minimum of area), then $\Sigma$ is a vertical plane.

This approach works well for surfaces that are smooth or regular with respect to the Riemannian structure on $\HH$, but there are important classes of surfaces in $\HH$ that are regular with respect to the sub-Riemannian structure but not the Riemannian structure. Once such class is defined in terms of the intrinsic gradient. 
Given a continuous function $f\from U\to \R$, we define $\nabla_f$ to be the vector field $\nabla_f=X-fZ$ on $U$. This is the push-forward of $X$ under the intrinsic projection from $\Gamma_f$ to $V_0$, i.e., $\nabla_f=(\Pi|_{\Gamma_f})_*(X)$. When $f$ is $C^1$, we can define the \emph{intrinsic gradient} of $f$ as $\nabla_ff$; when $f$ is merely $C^0$, one can define $\nabla_ff$ distributionally (see Section~\ref{sec:prelims}). If $\nabla_ff$ can be represented by a continuous function, we say that $f$ is of \emph{class $\COW$} or $f\in \COW(U)$.
We call the intrinsic graphs of $\COW$ functions \emph{$\COW$ graphs}. These are important examples of \emph{regular surfaces}; that is, they are the level sets of functions $\phi$ such that the horizontal gradient  $\nabla_{\mathbb{H}}\phi=(X\phi,Y\phi)$ is continuous and nonvanishing.

The $\COW$ condition bounds how quickly $f$ varies in the horizontal direction, but $f$ can vary very rapidly in non-horizontal directions. For example, Kirchheim and Serra Cassano constructed $\COW$ graphs that have Hausdorff dimension 2.5 with respect to the Euclidean metric~\cite{KirSC04}. This makes it difficult to apply standard variational methods to such surfaces, because perturbing a $\COW$ graph generally changes the set of horizontal directions. Indeed, a perturbation of a $\COW$ graph need not be $\COW$, and may even have infinite perimeter.

\begin{figure}
  \hfill
  \hfill
  \parbox{.5\textwidth}{\includegraphics[width=.5\textwidth]{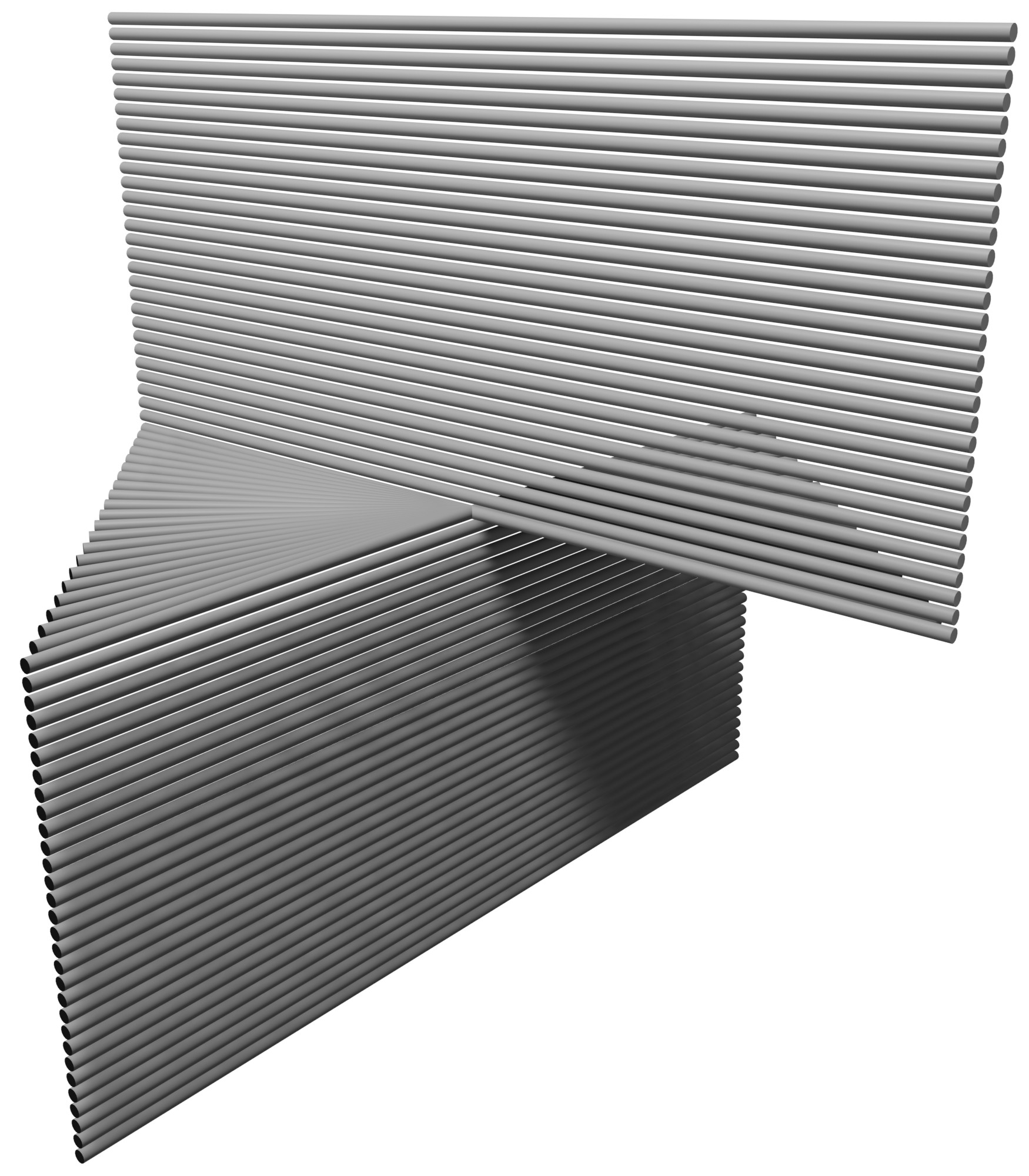}}
  \hfill
  \parbox{.3\textwidth}{\begin{center}
      \includegraphics[width=.27\textwidth]{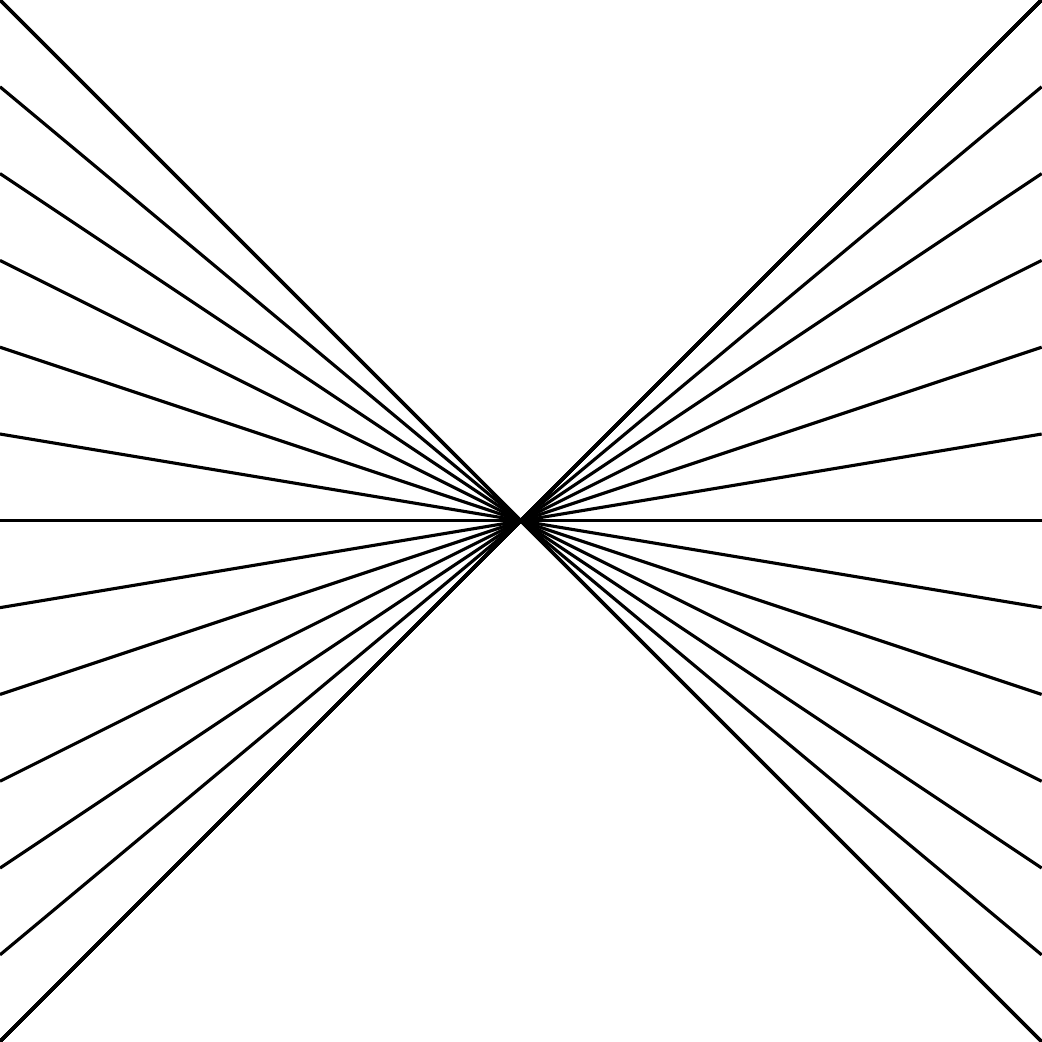}
      
      \bigskip 
      
      \includegraphics[width=.27\textwidth]{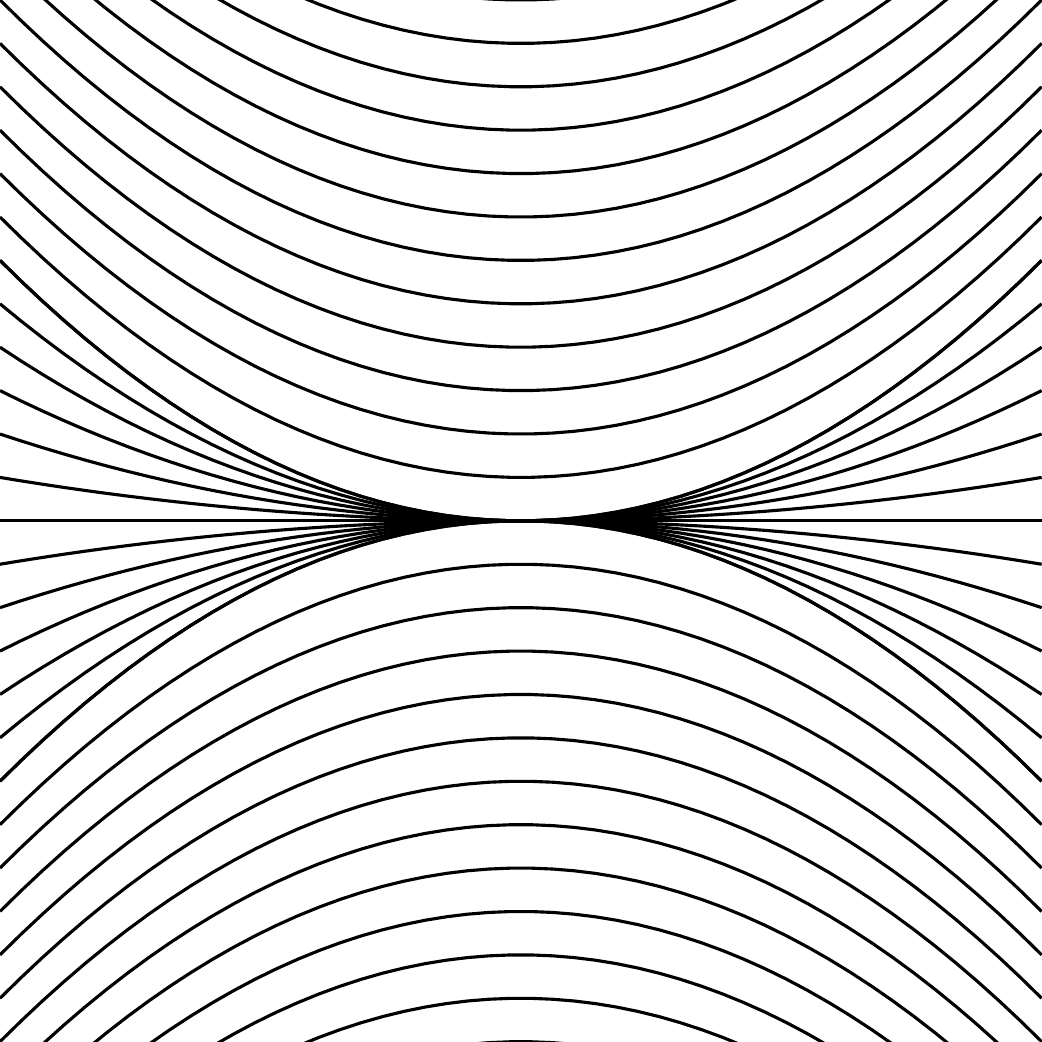}
    \end{center}}
  \hfill
  \hfill
  \caption{\label{fig:broken-plane}
    The left figure shows the broken plane $\BP_1\subset \HH$, which is made up of two half-planes connected by two wedges. (See Section~\ref{sec:scaling-limits} for the definition of $\BP_u$.) While the half-planes are foliated by horizontal lines, the horizontal lines that make up the wedges all intersect at the origin. The right figures show the projection of the horizontal lines in $\BP_1$ to the abelianization $\mathsf{A}$ and the projection along cosets of $\langle Y \rangle$ to the $xz$--plane $V_0$. Horizontal lines project to parabolas in $V_0$; the two half planes project to two families of parallel parabolas, while the wedges project to a family of parabolas through $\0$.  }
\end{figure}

This even affects relatively simple graphs, like the broken plane depicted in Figure~\ref{fig:broken-plane}. For $u\ge 0$, the broken plane $\BP_u$ is a surface made up of two vertical half-planes with slope $\pm u$ (the slope of a vertical half-plane is the slope of its projection to the $xy$--plane) connected by two wedges in the $xy$--plane. For any $u$, there is a function $b\from V_0\to \R$ such that $\Gamma_b=\BP_u$ and $b$ is intrinsic Lipschitz, that is, $\|\nabla_{b}b\|_\infty \le u$. This bounds how quickly $b$ varies in horizontal directions, but $b$ can still vary rapidly in vertical directions. In particular, $\partial_zb$ goes to infinity near $\0$, so if $f\in C^\infty_c(V_0)$ is nonzero at $\0$, then
$$\nabla_{b + f}[b+f] = (X-(b+f)Z)b + \nabla_{b + f}f= \nabla_b b - f\partial_z b + \nabla_{b + f}f.$$
Since $\nabla_bb$ and $\nabla_{b + f}f$ are bounded, this goes to infinity near $\0$, which means that $b+f$ is not intrinsic Lipschitz. This makes it difficult to decrease the area of $\BP_u$ by a smooth deformation. Indeed, Nicolussi Golo and Serra Cassano showed \cite{NSCBernstein} that $\BP_u$ is area-stable in the sense that $\partial_\epsilon[\area(\Gamma_{b+\epsilon f})]=0$ and $\partial^2_\epsilon[\area(\Gamma_{b+\epsilon f})]\ge 0$ for any $f\in C^\infty_c(V_0)$ and asked whether $\BP_u$ is area-minimizing.

One approach to questions like this is to consider variations of surfaces through intrinsic Lipschitz graphs. For example, \cite{Golo} and \cite{YoungHGraphs} studied variations through contact diffeomorphisms. These variations preserve the horizontal distribution on $\HH$ and send intrinsic Lipschitz graphs to intrinsic Lipschitz graphs, but there are many examples of surfaces that are area-stable with respect to contact variations but not smooth variations; Nicolussi Golo \cite{Golo} showed, for instance, that every ruled $\COW$ graph is area-stable with respect to contact variations.

The issue is that there are too few contact variations to recognize area-minimizing surfaces. Because contact variations preserve the horizontal distribution, they send horizontal curves to horizontal curves and do not affect the horizontal connectivity of a surface. That is, if $f\from \HH\to \HH$ is a contact diffeomorphism and $p$ and $q$ are connected by a horizontal curve in $\Sigma$, then $f(p)$ and $f(q)$ are connected by a horizontal curve in $f(\Sigma)$. Consequently, a surface that is area-minimizing with respect to contact variations can still have competitors with smaller area but different horizontal connectivity. Indeed, Figure~\ref{fig:conj-surfs} shows examples of competitors for $\BP_u$ with smaller area, but with a different pattern of horizontal curves. Finding a family of variations of non-smooth surfaces in $\HH$ that is rich enough to recognize area-minimizing surfaces would be an important step in the study of perimeter-minimizers in $\HH$.

In this paper, we introduce a family of variations of ruled surfaces and use it to characterize area-minimizing graphical strips. A \emph{graphical strip} is a ruled intrinsic graph that is symmetric around the $z$--axis (see Section~\ref{sec:deform-strips}). For example, the broken plane $\BP_u$ is a graphical strip. Previous work with graphical strips, for instance \cite{NSCBernstein}, \cite{GRBernstein}, and \cite{DGNPStrips}, mostly focused on the case of entire graphical strips, showing that, under some regularity conditions, an entire area-minimizing graphical strip is a vertical plane. Here, we will consider the case of graphical strips with boundary, especially graphical strips that project to a vertical strip in $V_0$.

We prove the following. Let $K=\{(x,0,z)\in V_0\mid -1\le x\le 1\}$ and let $U=\{(x,y,z)\in \HH\mid -1<x<1\}$. Let $\Sigma$ be a graphical strip over $K$. Then $\Sigma$ is an intrinsic graph, so we can define its epigraph $\Sigma^+$ (Section~\ref{sec:int-graphs}). We say that $\Sigma$ is area-minimizing in $U$ if $\Sigma^+$ is perimeter-minimizing in $U$.
\begin{thm}\label{thm:minimal-strips}
  Let $\Sigma$ be a graphical strip over $K$. Then $\Sigma$ is area-minimizing in $U$ if and only if there is a function $\sigma\from \R\to \R$ such that
  $-2 \le \frac{\sigma(t) - \sigma(s)}{t - s} < 2$ for all $s<t$ and
  $$\Sigma = \bigcup_{z\in \R} [(-1, -\sigma(z), z), (1, \sigma(z), z)],$$
  where $[p_1,p_2]$ denotes the line segment from $p_1$ to $p_2$.
\end{thm}

When there are $s<t$ such that $\frac{\sigma(t) - \sigma(s)}{t - s} \ge 2$, the surface $\Sigma$ is not an intrinsic graph. When $\sigma'(t) = -2$ for all $t$, $\Sigma$ is area-minimizing, but it is one of uncountably many area-minimizing surfaces with the same boundary.
\begin{thm}\label{thm:non-unique}
  Let $\gamma=\{(-1, 2z,z)\mid z\in \R\} \cup \{(1, -2z,z)\mid z\in \R\}$. Any two points $(-1, 2z_1,z_1)$ and $(1, -2z_2,z_2)$ are connected by a horizontal line. Let $\rho \from \R\to \R$ be a surjective continuous increasing function and let
  $$\Sigma_\rho = \bigcup_{z\in \R} [(-1, 2z,z), (1,-2\rho(z),\rho(z))].$$
  Then $\Sigma_\rho$ is an area-minimizing surface with $\partial \Sigma_\rho = \gamma$.
\end{thm}
Pauls gave an example of a closed curve that admits two different fillings by ruled surfaces in \cite{PaulsMinimalSurf}, showing that the Heisenberg minimal surface equation with Dirichlet boundary conditions can have multiple solutions, but the ruled surfaces he constructed are not area-minimizing \cite{ChengHwangYang}.

A further consequence of Theorem~\ref{thm:minimal-strips} is that if $\Sigma$ is an area-minimizing graphical strip over $K$, then no two rulings of $\Sigma$ intersect. It follows that when $u>0$, the broken plane $\BP_u$ is not area-minimizing.

\begin{figure}
  \hfill
  \hfill 
  \parbox{.5\textwidth}{\includegraphics[width=.5\textwidth]{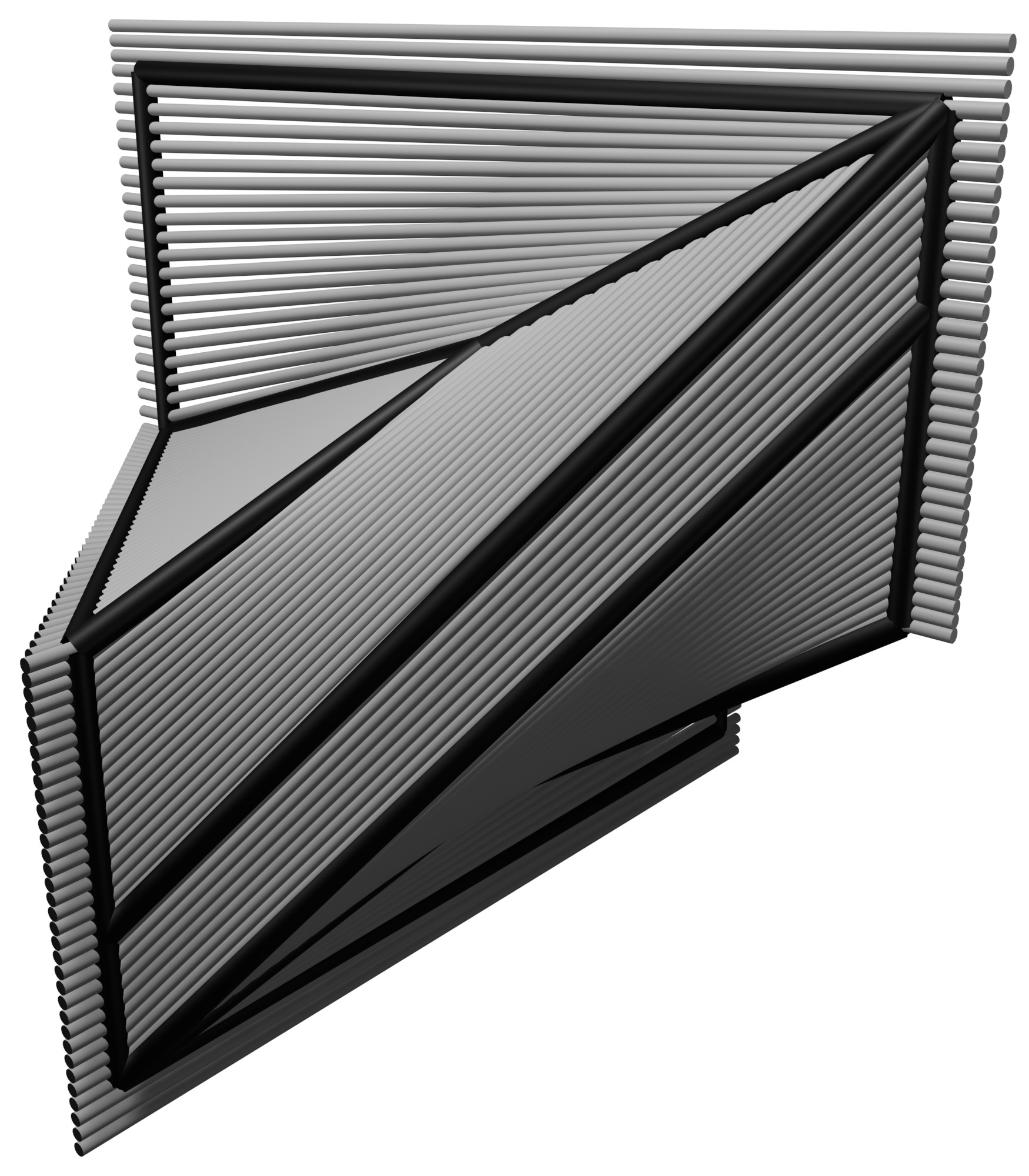}}
  \hfill 
  \parbox{.3\textwidth}{\begin{center}
      \includegraphics[width=.27\textwidth]{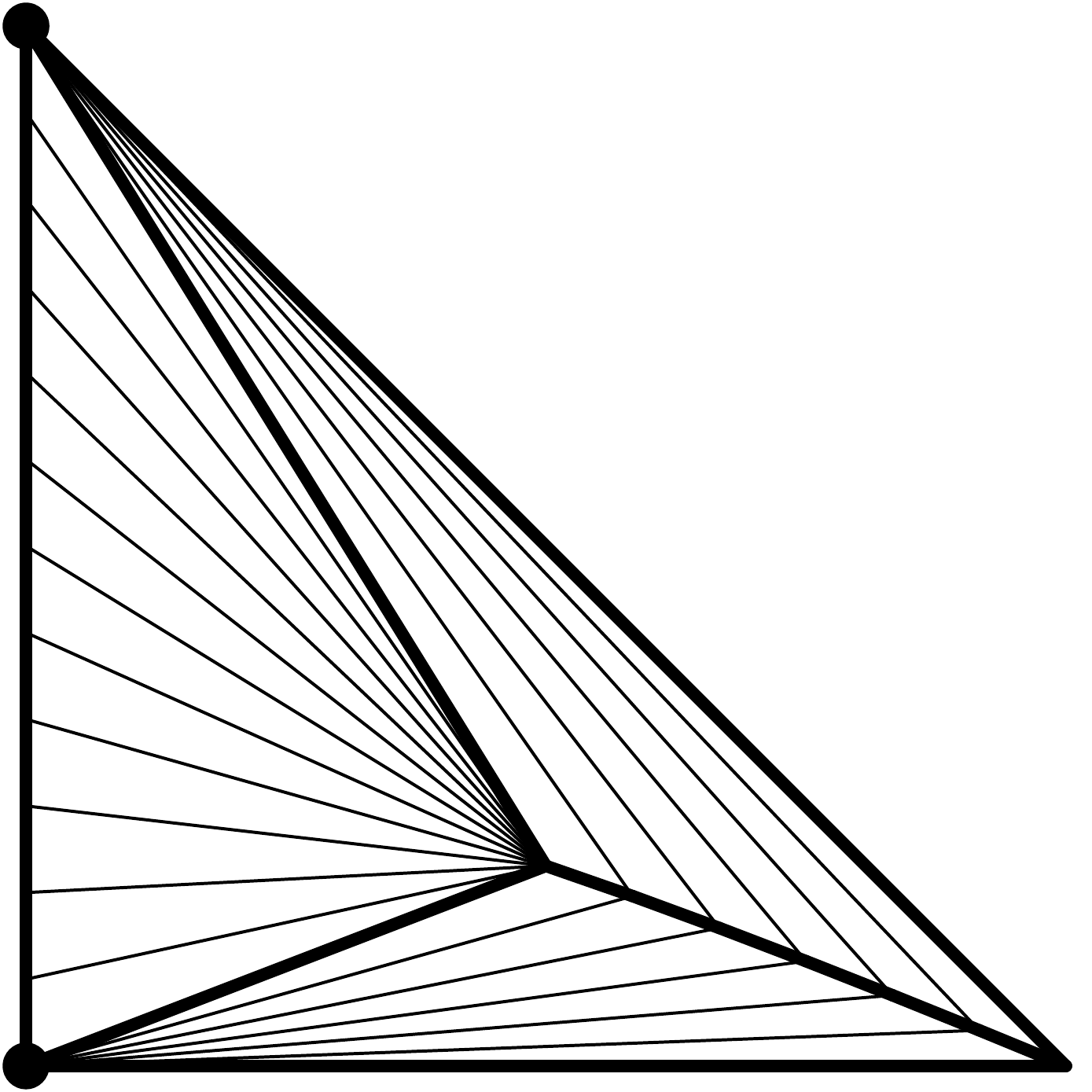}

      \bigskip 

      \includegraphics[width=.27\textwidth]{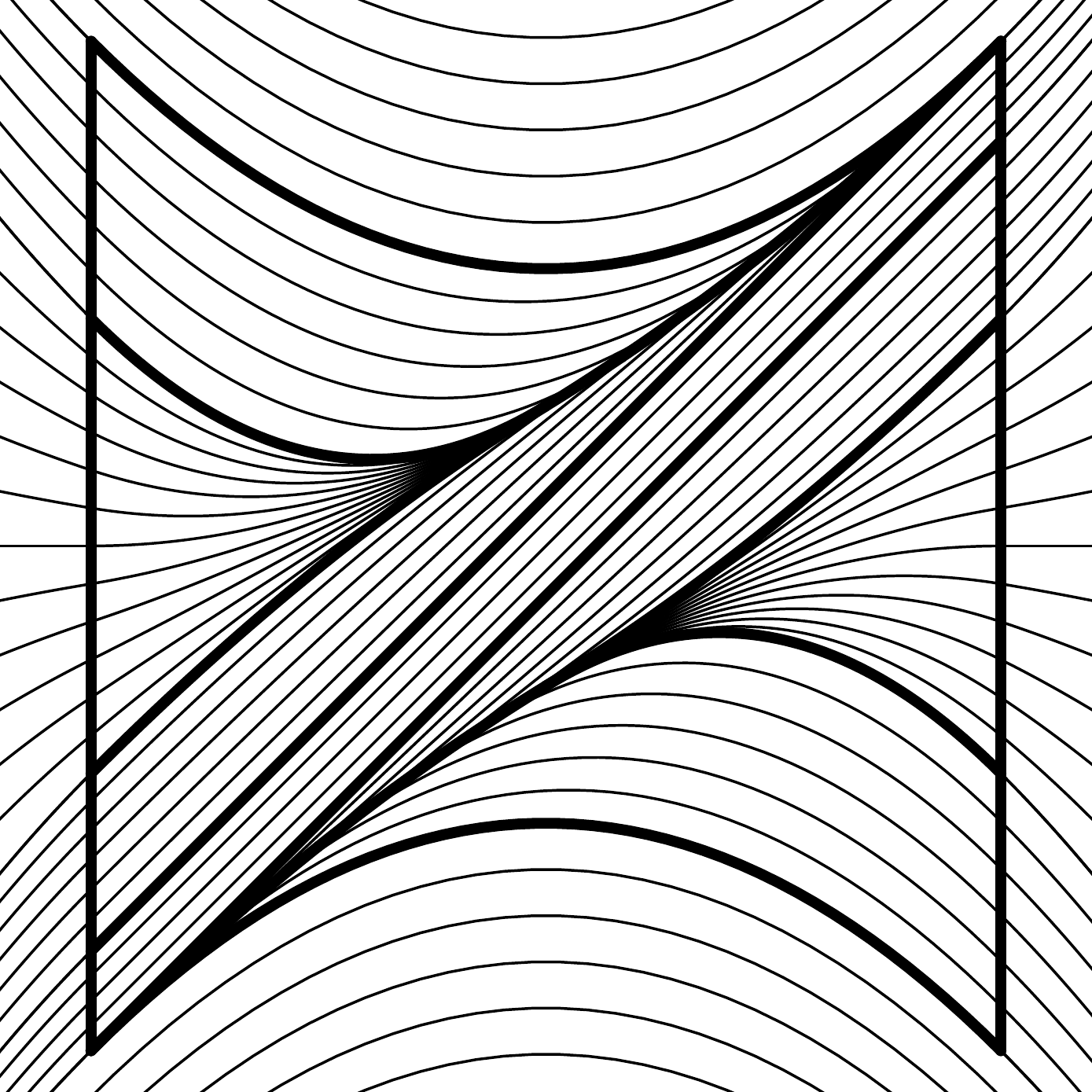}
    \end{center}}
  \hfill
  \hfill
  
  \bigskip 

  \hfill
  \hfill
  \parbox{.5\textwidth}{\includegraphics[width=.5\textwidth]{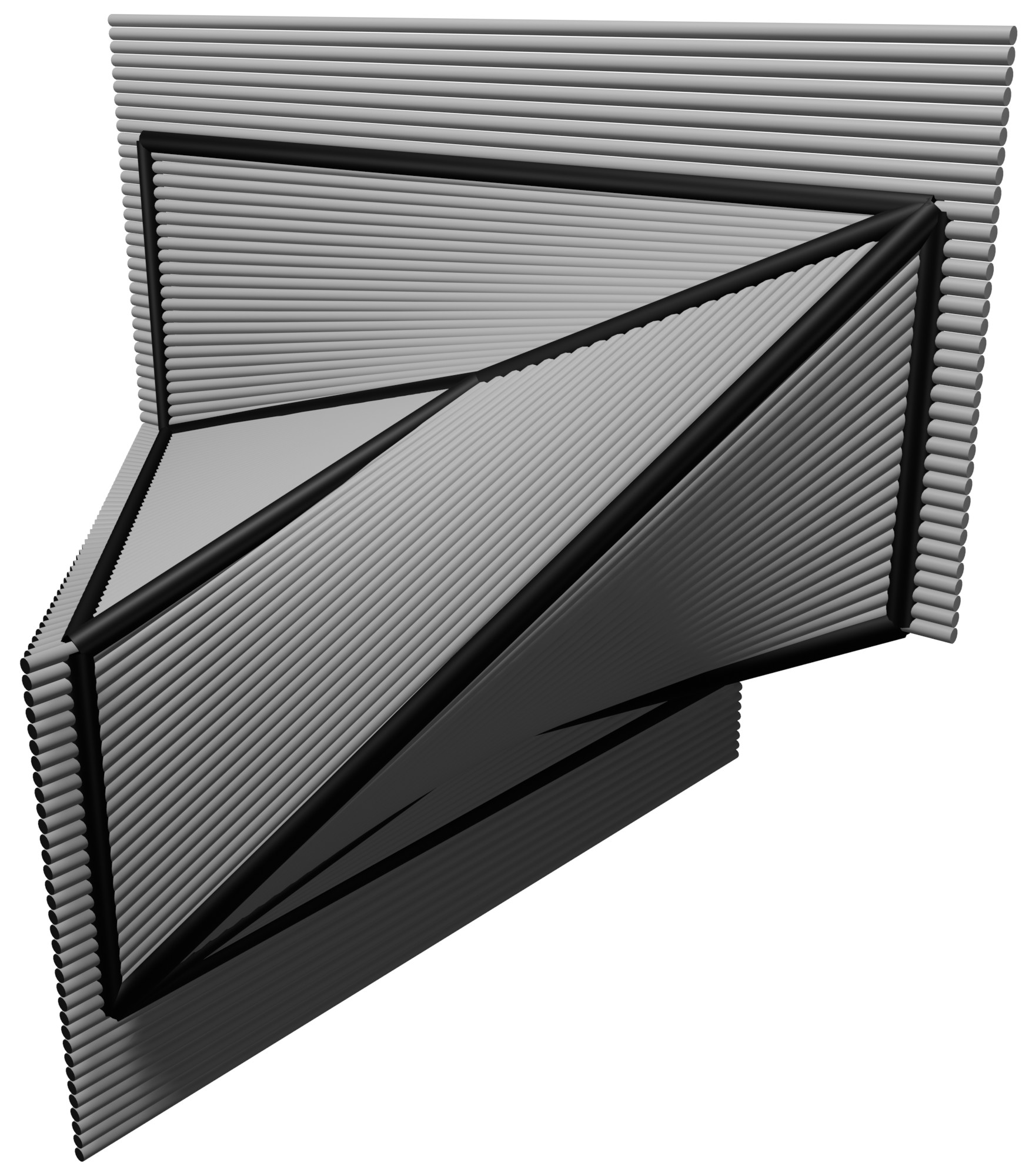}}
  \hfill 
  \parbox{.3\textwidth}{\begin{center}
      \includegraphics[width=.27\textwidth]{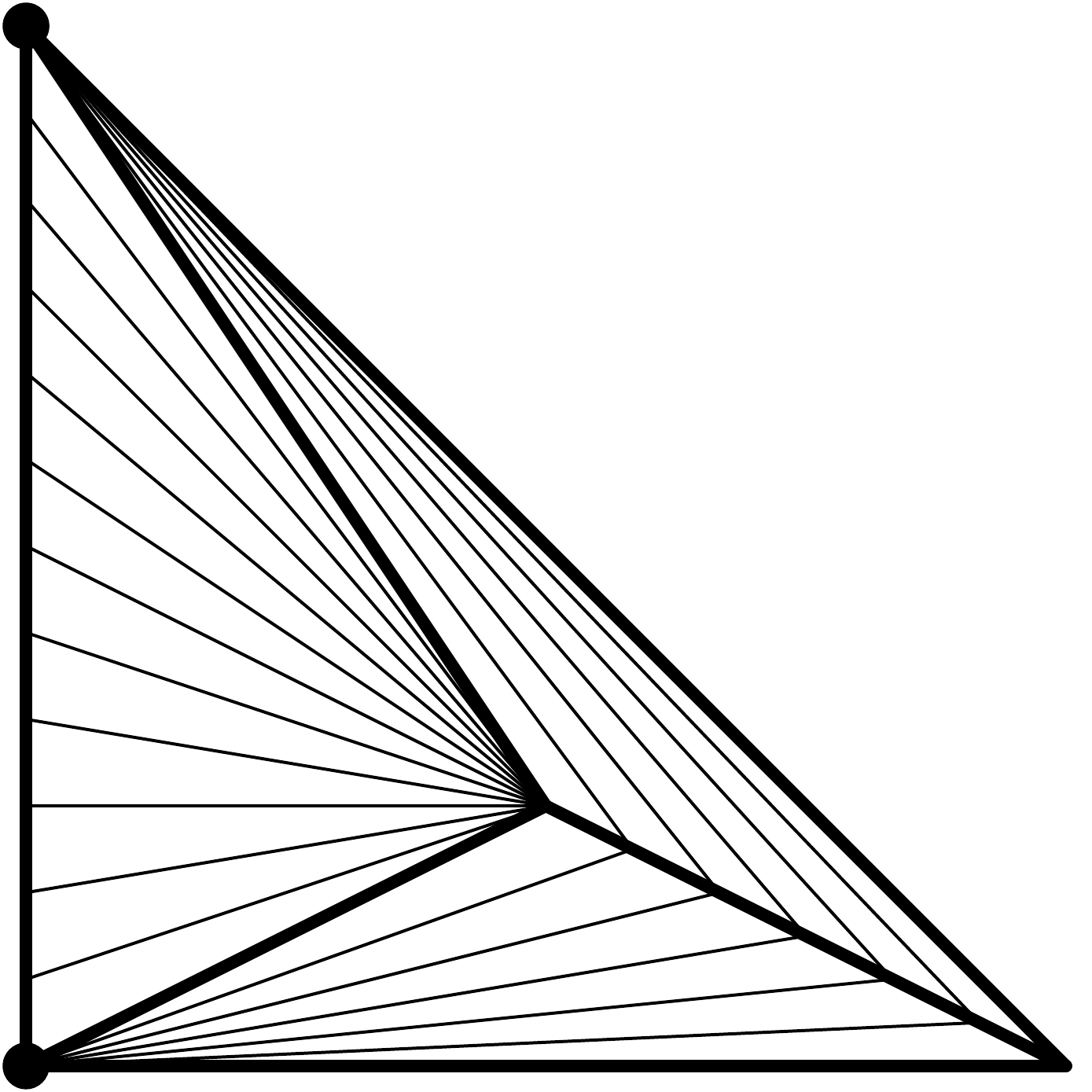}
      
      \bigskip 
      
      \includegraphics[width=.27\textwidth]{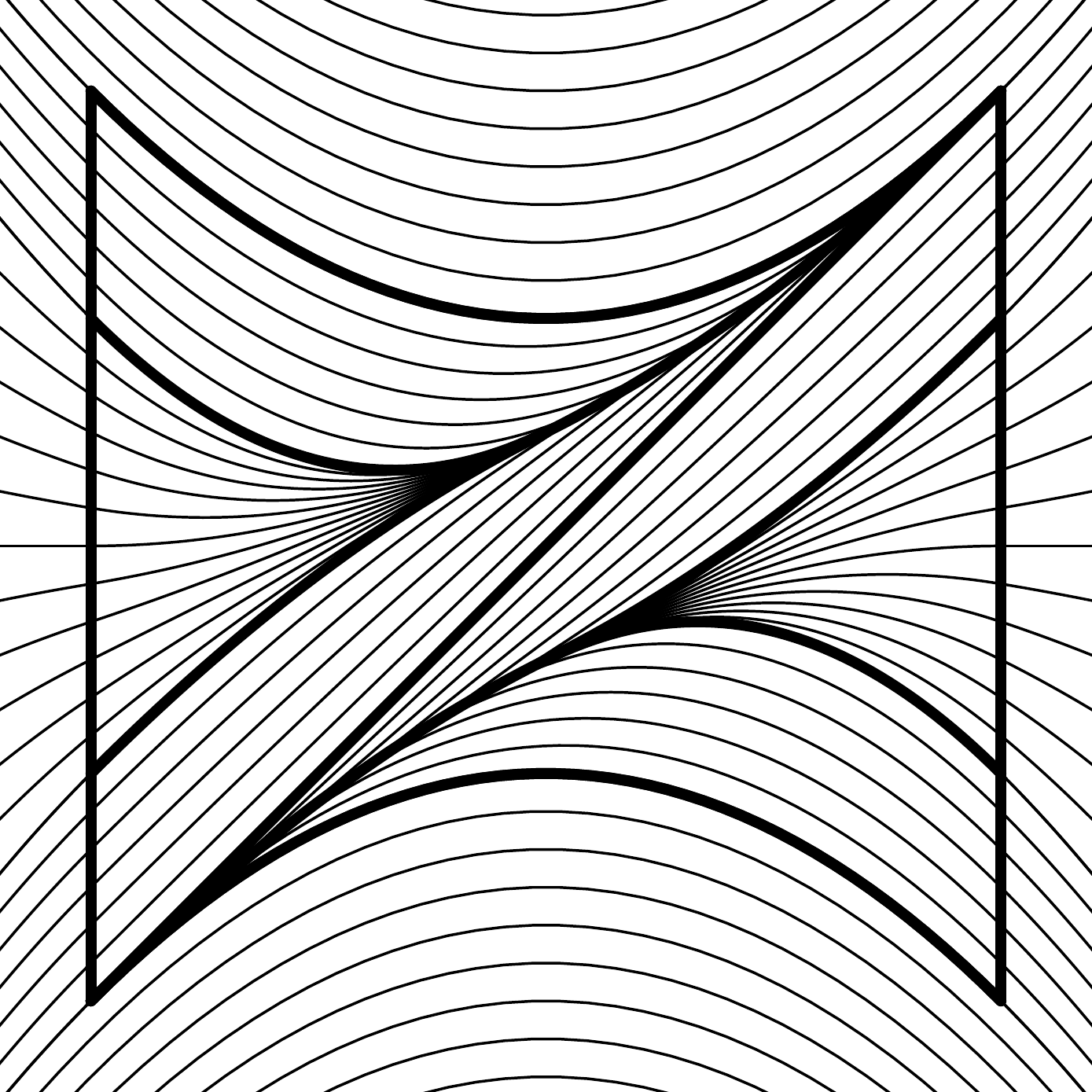}
    \end{center}}
  \hfill
  \hfill
  \caption{\label{fig:conj-surfs}
    Conjectured area-minimizing (top) and energy-minimizing (bottom) competitors for $\BP_1$. (See Section~\ref{sec:competitors} for full definitions.) The figures on the right are the projections to $\mathsf{A}$ and $V_0$. For clarity, the projections to $\mathsf{A}$ only show the top half of each surface. Black lines on the surfaces correspond to thick lines in the projections; the dots mark the images of vertical black lines.
  }
\end{figure}

The proof of Theorem~\ref{thm:minimal-strips} relies on constructing deformations of graphical strips through piecewise-ruled surfaces and computing the second variation of the area under such deformations (Proposition~\ref{prop:vf-strips}). We construct these deformations as follows. Since $\Sigma$ is an intrinsic graph over $K$ which is symmetric around the $z$--axis, the boundary of $\Sigma$ consists of two intrinsic graphs $X \gamma_\alpha$ and $X^{-1} \gamma_{-\alpha}$, where $\alpha\from \R\to \R$ is continuous and $\gamma_\alpha:=\{(0,y,z)\mid y=\alpha(z)\}$ is the graph of $\alpha$ in the $yz$--plane.

Cutting $\Sigma$ along the $z$--axis produces two symmetric halves, one bounded by $\langle Z\rangle$ and $X \gamma_\alpha$ and one bounded by $\langle Z\rangle$ and $X^{-1} \gamma_{-\alpha}$. Given a function $\tau\in C^\infty_c(\R)$, we construct a surface $S_{\alpha,\tau}$ consisting of two ruled surfaces, one bounded by $\gamma_\tau$ and $X \gamma_\alpha$ and one bounded by $\gamma_\tau$ and $X^{-1} \gamma_{-\alpha}$. In  Section~\ref{sec:deform-strips}, we prove Theorem~\ref{thm:minimal-strips} by expanding $\area S_{\alpha, \lambda \tau}$ to second order in $\lambda$.

In Sections~\ref{sec:scaling-limits}--\ref{sec:proof-ruled-bernstein}, we apply Theorem~\ref{thm:minimal-strips} to show that Bernstein's theorem holds for the class of ruled intrinsic graphs.
\begin{thm}\label{thm:ruled-bernstein}
  An entire area-minimizing ruled intrinsic graph in $\HH$ is a vertical plane.
\end{thm}
We prove Theorem~\ref{thm:ruled-bernstein} by showing that if $\Gamma$ is a ruled intrinsic graph, then its scaling limit is a plane or a broken plane (Section~\ref{sec:scaling-limits}). Scaling limits of perimeter-minimizing sets are perimeter-minimizing (Section~\ref{sec:proof-ruled-bernstein}), so if $\Gamma$ satisfies Theorem~\ref{thm:ruled-bernstein}, then its scaling limit is a plane. It then follows that $\Gamma$ itself is a plane.

Finally, in Section~\ref{sec:non-unique}, we prove Theorem~\ref{thm:non-unique}, and in Section~\ref{sec:competitors} we construct conjectural area-minimizing and energy-minimizing competitors for $\BP_u$, seen in Figure~\ref{fig:conj-surfs}. These competitors are each made up of two $Z$--graphs (graphs of an equation $z=f(x,y)$); one can show numerically that these surfaces have smaller area than $\BP_u$, but we do not know whether they minimize the area or the energy.

The question of finding the broadest class of surfaces in $\HH$ that satisfies Bernstein's theorem remains open. 
It is known that intrinsic Lipschitz graphs do not satisfy Bernstein's theorem; there are many area-minimizing entire intrinsic Lipschitz graphs that are not vertical planes. These can be highly singular; the first example of such a graph had a singularity along the $x$--axis \cite{PaulsHMinimal}, and the examples in \cite{GR2020} can be chosen to have a characteristic set (the set of points where the tangent plane to the surface is horizontal) whose closure has positive measure. In these examples, however, the surface fails to be $\COW$ near the singularities, so the Bernstein problem for $\COW$ graphs is an open question. 

\begin{remark}
  The 3D models used in the figures in this paper can be found in .obj format in the ancillary files on the arXiv page for this paper. These files can be opened in Preview on Macs, Paint 3D on Windows, or any 3D modeling program. They can also be found at \url{https://cims.nyu.edu/~ryoung/ruledBernstein/}.
\end{remark}

\noindent\textbf{Acknowledgments:} This material is based upon work supported by the National Science Foundation under Grant Nos.\ 2005609 and 1926686 and research done while the author was a visiting member at the Institute of Advanced Study. The author would like to thank Sebastiano Nicolussi Golo, Manuel Ritoré, and Richard Schwartz for their time and advice during the preparation of this paper and to thank the Institute of Advanced Study for its hospitality.

\section{Preliminaries and notation}\label{sec:prelims}

\subsection{The Heisenberg group}
The Heisenberg group $\HH$ is the $3$--dimensional simply connected Lie group with Lie algebra
$$\mathfrak{h}:=\langle X, Y, Z\mid [X,Y]=Z, [X,Z]=[Y,Z]=0\rangle.$$
We identify $\HH$ with $\mathfrak{h}$ via the Baker--Campbell--Hausdorff formula, i.e., $\HH=\langle X,Y,Z\rangle \cong \R^3$ and 
$$(x,y,z)\cdot(x',y',z')=\left(x+x',y+y',z+z'+\frac{xy'-yx'}{2}\right).$$
We use $X$, $Y$, and $Z$ to denote the coordinate vectors of $\R^3$ and the corresponding left-invariant fields $X_{(x,y,z)}=(1,0,-\frac{y}{2})$, $Y_{(x,y,z)}=(0,1,\frac{x}{2})$, and $Z_{(x,y,z)}=(0,0,1)$. Let $x,y,z\from\HH\to \R$ be the coordinate functions of $\HH$. Every vector $v\in \HH$ generates a one-parameter subgroup of $\HH$; we write $\langle v\rangle=\R v$ for this subgroup and define $v^t=tv$ for all $t\in \R$.

We equip $\HH$ with the Korányi norm $\|(x,y,z)\|_{\Kor}=\sqrt[4]{(x^2+y^2)^2+z^2}$ and the corresponding left-invariant distance $d(p,q)=\|p^{-1}q\|_{\Kor}$. For $a,b\in \R\setminus \{0\}$, let $s_{a,b}$ be the automorphism $s_{a,b}(x,y,z)=(ax,by,abz)$. We call automorphisms of the form $s_{t,t}$, $t>0$ \emph{scaling automorphisms}; these satisfy $\|s_{t,t}(p)\|_{\Kor}=t\|p\|_{\Kor}$ and $d(s_{t,t}(p),s_{t,t}(q))=td(p,q)$. For $p\in \HH$ and $r>0$, let $B(p,r):=\{q\in \HH\mid d(p,q)< r\}$ be the open ball of radius $r$ around $p$. The tangent planes spanned by $X$ and $Y$ are called the \emph{horizontal distribution}, and vectors in the horizontal distribution are called \emph{horizontal vectors}. A curve $\gamma\from I\to \HH$ whose coordinates are Lipschitz and such that $\gamma'(t)$ is a horizontal vector for almost every $t$ is called a \emph{horizontal curve}.

\subsection{The perimeter measure}
The \emph{sub-Riemannian perimeter} of a measurable subset $E\subset \HH$ on an open set $\Omega\subset \HH$ is given by
$$\Per_{E}(\Omega)=\sup\left\{ \int_E \div_{\mathbb{H}}(a X + bY) \ud \eta \mid a,b\in C^\infty_c(\Omega), a^2+b^2\le 1 \right\},$$
where the \emph{horizontal divergence} $\div_{\mathbb{H}}$ is defined by $\div_{\mathbb{H}}(a X+b Y) = Xa + Yb$ and $\eta$ is Lebesgue measure on $\HH$. This perimeter was introduced in \cite{FSSCRectifiability} as a Heisenberg analogue of perimeter in Euclidean space.  If $\Per_E(B(\0,r))<\infty$ for every $r>0$, we say that $E$ has \emph{locally finite perimeter}. The perimeter is lower semicontinuous, i.e., if $E_1,E_2,\dots$ have locally finite perimeter and $\one_{E_i}$ converges locally in $L_1$ to $\one_E$, then $E$ has locally finite perimeter and $\Per_{E}(\Omega)\le \liminf_{i\to \infty} \Per_{E_i}(\Omega)$ \cite[2.12]{FSSCRectifiability}.

Let $E\symdiff F:= (E\setminus F)\cup (F\setminus E)$ be the symmetric difference operator. For $E\subset \HH$ a measurable set and $\Omega\subset \HH$ an open set, we say that $E$ is a \emph{perimeter minimizer} in $\Omega$ if for every $r>0$ and every measurable $F\subset \HH$ such that satisfies $E\symdiff F\Subset B(\0,r)\cap \Omega$, we have $\Per_E(B(\0,r)\cap \Omega)\le \Per_F(B(\0,r)\cap \Omega)$. 

\subsection{Intrinsic graphs and intrinsic gradient}\label{sec:int-graphs}

For any $U\subset V_0$ and any function $f\from U\to \R$, we define $\Gamma_f=\{uY^{f(u)}\mid u\in U\}$ to be the \emph{intrinsic graph} of $f$ and $\Gamma^+_f=\{uY^{t}\mid t>f(u)\}$ to be the \emph{epigraph} of $f$. The intrinsic graph $\Gamma_f$ is parametrized by the function $\Psi_f\from U\to \HH$, $\Psi_f(u)=uY^{f(u)}$, and we define $\Pi(p)=pY^{-y(p)}$,
$$\Pi(x,y,z)=\left(x,0,z-\frac{xy}{2}\right),$$
to be the intrinsic projection from $\HH$ to $V_0$, so that $\Pi\circ\Psi_f=\id_U$.

Let $\nabla_f$ be the vector field $\nabla_f=X-fZ=(\Pi|_{\Gamma_f})_*(X)$ on $U$. We call this the \emph{intrinsic gradient}; we are particularly interested in $\nabla_ff$, which determines the tangent plane to $\Gamma_f$. The derivative $\nabla_ff$ exists when $f$ is $C^1$, but it can be defined distributionally when $f$ is merely $C^0$. 
If $f\from V_0\to \R$ is continuous, we say that $\nabla_ff$ exists in the sense of distributions if there is a function $\theta \in L_{\infty,\mathrm{loc}}$ such that for every $\psi\in C^1_c$,
$$\int_{V_0} \theta \psi \ud \mu= \int_{V_0} -f  \partial_x\psi + \frac{f^2}{2}\partial_z\psi\ud \mu,$$
where $\mu$ is Lebesgue measure on $V_0$.
If so, we write $\nabla_ff=\theta$.  When $f$ is $C^1$, this coincides with the previous definition.
For $U$ an open subset of $V_0$, we define $\COW(U)$ to be the set of continuous functions $f\from U\to \R$ such that $\nabla_ff$ is represented by a continuous function on $U$. This implies that $\Gamma_f$ is a regular surface in the sense of \cite{AVSCIntrinsic}.

One can also define a version of the Lipschitz condition adapted to the Heisenberg group. For $D\subset V_0$, we say that a function $f\from D\to \R$ is \emph{intrinsic Lipschitz} or that $\Gamma_f$ is an \emph{intrinsic Lipschitz graph} if there exists a $0<\lambda<1$ such that
$|y(p)-y(q)|\le \lambda d(p,q)$
for all $p,q\in \Gamma_f$. By Theorem~4.29 of \cite{FSSCDifferentiability}, if $U$ is an open set and $f\from U\to \R$ is an intrinsic Lipschitz function, then $f$ satisfies an intrinsic version of Rademacher's theorem in the sense that $\nabla_ff$ exists in the sense of distributions and $\|\nabla_ff\|_\infty$ is bounded by a function of $\lambda$. Conversely, if $f\in C^0(U)$ and $\nabla_ff\in L_\infty(U)$, then $f$ is locally intrinsic Lipschitz \cite{BiCaSC}.

Let $D\subset V_0$ and let $f\from D\to \R$ be an intrinsic Lipschitz graph. The area formula \cite[Thm.\ 1.6]{CMPSC} states that for any bounded open set $U\subset V_0$,
$$\Per_{\Gamma^+_f}(\Pi^{-1}(U))=\int_U \sqrt{1+(\nabla_f f)^2} \ud \mu.$$
More generally, we define
\begin{equation}\label{eq:area-formula}
  \area \Psi_f(W):=\int_W \sqrt{1+(\nabla_f f)^2} \ud \mu.
\end{equation}
for any measurable $W\subset D$. 
For a continuous function $f\from V_0\to \R$ and an open subset $\Omega\subset \HH$, we say that $\Gamma_f$ is \emph{area-minimizing} in $\Omega$ if $\Gamma_f^+$ is perimeter-minimizing in $\Pi^{-1}(\Omega)$. 

\subsection{Horizontal lines and characteristic curves}\label{sec:char}

Let $\pi\from \HH\to \mathsf{A}$, $\pi(x,y,z)=(x,y,0)$. 
A \emph{horizontal line} is a coset of the form $p \langle V\rangle$, where $V\in \mathsf{A}$. The \emph{slope} of a horizontal line $L$ is the slope of $\pi(L)$, i.e., $\slope(p \langle aX + bY\rangle)=\frac{b}{a}$.

Let $p=(x_p,y_p,z_p) \in \HH$ and $m\in \R$, and let $L=p\langle X+mY\rangle$ be the line of slope $m$ through $p$. Let $\lambda(t)=p(X+mY)^{t-x_p}$ parametrize $L$. Then
\begin{align}\label{eq:point-slope}
\notag  \Pi(\lambda(x))& = \Pi\left(p(X+mY)^{x-x_p}\right) = \Pi\left(x, y_p + m(x-x_p), z_p +\frac{x_p m(x-x_p) - y_p (x-x_p)}{2}\right)\\
\notag                 &= \left(x, 0, z_p - \frac{x_py_p}{2} - y_p (x-x_p) - \frac{m}{2}(x-x_p)^2\right) \\
                 & =
                   \Pi(p) + \left(x-x_p, 0, - y_p (x-x_p) - \frac{m}{2}(x-x_p)^2\right).
\end{align}
That is, $\Pi(L)$ is a parabola, and any parabola in $V_0$ is the projection of a unique horizontal line.

Given an open subset $U\subset V_0$ and a continuous function $f\from U\to \R$, the \emph{characteristic curves} of $\Gamma_f$ are the integral curves of the vector field $\nabla_f=X-fZ$.
By the Peano Existence Theorem, for every $p\in U$, there is a characteristic curve through $p$, but when $f$ is not Lipschitz, this curve may not be unique. Note, however, that two characteristic curves that meet at a point $p$ must have the same tangent vector at $p$.

The characteristic curves of $\Gamma_f$ are projections of horizontal curves; if $f\from V_0\to \R$ is an intrinsic Lipschitz function and $\lambda\from I\to V_0$ is a characteristic curve for $\Gamma_f$, then $\gamma=\Psi_f\circ \lambda$ is a horizontal curve in $\Gamma_f$; conversely, if $\gamma\from I\to \Gamma_f$ is a horizontal curve such that $x(\gamma(t))=t$ for all $t$, then $\Pi\circ \gamma$ is a characteristic curve.

When $\Gamma_f$ is a ruled surface, we can say more about its characteristic curves. By \eqref{eq:point-slope}, the projection of any ruling of $\Gamma_f$ is a parabola. If $R_1$ and $R_2$ are distinct rulings of $\Gamma_f$ such that $\Pi(R_1)$ and $\Pi(R_2)$ intersect at a point $p\in V_0$, then $\Pi(R_1)$ and $\Pi(R_2)$ must be tangent at $p$. It follows that $\Pi(R_1)$ and $\Pi(R_2)$ cannot cross. That is, the following lemma holds.
\begin{lemma}\label{lem:ordering-rulings}
  Let $R_1$ and $R_2$ be distinct rulings of \(\Gamma_f\). Let $g_{R_i}\from I_i\to \R$ be such that $\Pi(R_i)=\{(x,0,g_{R_i}(x))\mid x\in I_i\}$. Then $g_{R_1}(x) \le g_{R_2}(x)$ for all $x\in I_1\cap I_2$ or $g_{R_1}(x) \ge g_{R_2}(x)$ for all $x\in I_1\cap I_2$.
\end{lemma}

\section{Deformations of graphical strips}\label{sec:deform-strips}

Let $D\subset V_0$. A \emph{graphical strip} over $D$ is an intrinsic graph $\Gamma_f$ of a continuous function $f\from D\to \R$ such that $\Gamma_f$ is ruled and every ruling intersects the $z$--axis. (This definition differs slightly from the definitions found in \cite{NSCBernstein}, \cite{GRBernstein}, and \cite{DGNPStrips}, which assume that $f$ is Lipschitz or $C^2$ with respect to the Euclidean structure on $\R^2$.)

We mainly consider graphical strips $\Gamma$ over $K=\{(x,0,z)\in V_0\mid -1\le x\le 1\}$, in which case $\Gamma$ is symmetric around the $z$--axis (i.e., $s_{-1,-1}(\Gamma)=\Gamma$) and $\Gamma$ is determined by the intersection $\Gamma\cap \{x=1\}$. That is,
$$\Gamma=\bigcup_{p\in \Gamma\cap \{x=1\}} [p,s_{-1,-1}(p)],$$
where $[p_1,p_2]$ denotes the line segment from $p_1$ to $p_2$.

Let $\Gamma_f$ be a graphical strip over $K$ and let $\alpha\from \R\to \R$, $\alpha(w)=f(1,0,w)$. Then
$$\Gamma_f\cap \{x=1\}=X\gamma_\alpha,$$
where $\gamma_\alpha:=\{(0, \alpha(z) , z)\mid z\in \R\}$ is the graph of $\alpha$ in the $yz$--plane, and $\Gamma_f=S_\alpha$, where
\begin{align*}
  S_\alpha 
  &:= \bigcup_{w\in \R} \left[X^{-1} \cdot (0,-\alpha(w), w), X \cdot (0,\alpha(w), w)\right] \\ 
  &= \bigcup_{w\in \R} \left[\left(-1,-\alpha(w), w+\frac{\alpha(w)}{2}\right), \left(1,\alpha(w), w+\frac{\alpha(w)}{2}\right)\right].
\end{align*}
For $w\in \R$, let $\eta(w):= w+\frac{\alpha(w)}{2}$.

In general, for any continuous $\alpha\from \R\to \R$, we can define $S_\alpha$ as above. This is a ruled surface bounded by $X\gamma_\alpha$ and $X^{-1}\gamma_{-\alpha}$, but it is not always an intrinsic graph. In Section~\ref{sec:strip-prelims}, we will show the following lemma.
\begin{lemma}\label{lem:strip-injectivity}
  Let $S_\alpha$ be as above. Then $S_\alpha$ is a graphical strip over $K$ if and only if $\frac{\alpha(w_2)-\alpha(w_1)}{w_2-w_1} \ge -2$ for all $w_1<w_2$ and $\eta(\R)=\R$. 
\end{lemma}
Note in particular that if $S_\alpha$ is a graphical strip, then $\eta$ is a non-decreasing function; $\eta$ is increasing if and only if the rulings of $S_\alpha$ do not intersect.

A notable example of a graphical strip is the broken plane illustrated in Figure~\ref{fig:broken-plane}.
\begin{definition}\label{def:broken-plane}
  Let $u\ge 0$. 
  The \emph{broken plane} $\BP_u$ consists of two vertical half-planes of slope $\pm u$ connected by two wedges in the $xy$--plane. (See Figure~\ref{fig:broken-plane}.) That is,
  $$\BP_u := \{(x, -ux ,z)\mid x\in \R, z\ge 0\}
    \cup \{(x, ux ,z)\mid x\in \R, z\le 0\}
    \cup \{(x,y,0)\mid |y| \le u|x|\}.$$
  As $u\to \infty$, this converges to the set
  $$\BP_\infty := \langle Y,Z\rangle \cup \langle X, Y\rangle$$
\end{definition}

When $0<u<\infty$, $\BP_u=\Gamma_{b_u}$, where
$$b_u(x,0,z) = \begin{cases}
  u x & z < - \frac{u}{2} x^2 \\
  \frac{- 2 z}{x}& |z| \le \frac{u}{2} x^2 \\
  - u x & z > \frac{u}{2} x^2,
\end{cases}$$
and we define $\BP^+_u=\Gamma_{b_u}^+$. This is the union of two quadrants in $\HH$, one bounded by the $xy$--plane and the upper half-plane with slope $-u$ and one bounded by the $xy$--plane and the lower half-plane with slope $u$.

As $u\to \infty$, these epigraphs converge to the set
$$\BP_\infty^+ := \{(x,y,z)\in \HH\mid xz > 0\}.$$  

For any $0 \le u < \infty$, $\BP_u\cap \Pi^{-1}(K)$ is a graphical strip over $K$, and $\BP_u\cap \Pi^{-1}(K)=S_\alpha$, where
$$\alpha(z) = b_u(1,0,z) = \begin{cases}
  u & z < - \frac{u}{2} \\
  -2 z & |z| \le \frac{u}{2} \\
  -u & z > \frac{u}{2}.
\end{cases}$$

The main goal of this section is to prove the following  characterization of area-minimizing graphical strips over $K$. Let $U$ be the interior of $\Pi^{-1}(K)$, i.e., $U=\{(x,y,z)\in \HH\mid -1<x<1\}$. 
\begin{prop}\label{prop:strips}
  Let $S_\alpha$ be a graphical strip over $K$. Then $S_\alpha$ is area-minimizing on $U$ if and only if $\frac{\alpha(w_2)-\alpha(w_1)}{w_2-w_1} \ge -1$ for all $w_1<w_2$. 
\end{prop}

By the following lemma, this characterization is equivalent to the characterization in Theorem~\ref{thm:minimal-strips}.
\begin{lemma}\label{lem:alternate}
  Let $S_\alpha$ be a graphical strip over $K$. Then $\frac{\alpha(w_2)-\alpha(w_1)}{w_2-w_1} \ge -1$ for all $w_1<w_2$ if and only if there is a function $\sigma\from \R\to \R$ such that
  $-2 \le \frac{\sigma(z_2) - \sigma(z_1)}{z_2 - z_1} < 2$ for all $z_1<z_2$ and
  $$S_\alpha = \bigcup_{z\in \R} \big[(-1, -\sigma(z), z), (1, \sigma(z), z)\big].$$
\end{lemma} 
We will prove Lemma~\ref{lem:alternate} in  Section~\ref{sec:strip-prelims}.

There are two main ideas to the proof of Proposition~\ref{prop:strips}. First, in Section~\ref{sec:deforming-ruled}, we will show that a perimeter-minimizing graphical strip satisfies $\frac{\alpha(w_2)-\alpha(w_1)}{w_2-w_1} \ge -1$ for all $w_1<w_2$ by defining a family $S_{\alpha,\tau}$ of deformations of $S_\alpha$. These deformations are parametrized by functions $\tau\in C^\infty_c(\R)$. Each surface $S_{\alpha,\tau}$ is a union of a ruled graph over the strip $K^+=[0,1] \times \{0\} \times \R$ and a ruled graph over the strip $K^-=[-1,0] \times \{0\} \times \R$, with common boundary $\gamma_\tau$. We will show that the area of these deformations satisfies the following formula. 
\begin{prop}\label{prop:vf-strips}
  Let $S_\alpha$ be a graphical strip over $K$ and suppose that $\alpha\from \R\to \R$ is Lipschitz.
  
  Let $w_1<w_2$ be such that $\eta^{-1}(\eta(w_i))=\{w_i\}$ and let $D=\{(x,y,z) \in \HH\mid \eta(w _1)\le z\le \eta(w _2) \}$. Let $\tau\in C^\infty_c([\eta(w_1),\eta(w_2)])$. Then there is a $C=C(\tau)>0$ such that for all $\lambda \in (-C^{-1},C^{-1})$, 
%  $$\left| \area (D \cap S_{\alpha,\lambda \tau}) - \area (D \cap S_{\alpha,\lambda \tau}) - \lambda^2 \int_{w_1}^{w_2} \frac{\tau(\eta(w))^2}{(1+\alpha(w)^2)^{\frac{3}{2}}}(1-\alpha'(w))\ud w\right|\le C \lambda^3.$$
  \begin{equation}\label{eq:strip-var}
    \left| \area (D \cap S_{\alpha,\lambda \tau}) - \area (D \cap S_{\alpha}) - \lambda^2 \II(\tau)\right|\le C (w_2-w_1) |\lambda|^3,
  \end{equation}
  where
  \begin{align*}
    \II(\tau) 
    &= \int_{-\infty}^{\infty} \frac{\tau(\eta(w))^2}{(1+\alpha(w)^2)^{\frac{3}{2}}}(1+\alpha'(w))\ud w 
  \end{align*}
  % \\
  %   &= \int_{\eta(w_1)}^{\eta(w_2)} \frac{\tau(z)^2}{(1+\sigma(z)^2)^{\frac{3}{2}}}\left(1+\frac{\sigma'(z)}{2}\right)\ud z\\
  %   &= \int_{\eta(w_1)}^{\eta(w_2)} \frac{\tau(z)^2}{(1+\sigma(z)^2)^{\frac{3}{2}}}\ud z + \frac{1}{2} \int_{\eta(w_1)}^{\eta(w_2)} \frac{\tau(z)^2}{(1+\sigma(z)^2)^{\frac{3}{2}}} \ud \sigma(z).
  % \end{align*}
\end{prop}

In Lemma~\ref{lem:min-implies-bound}, we conclude that if $S_\alpha$ is an area-minimizing graphical strip over $K$, then $\frac{\alpha(w_2)-\alpha(w_1)}{w_2-w_1} \ge -1$ for all $w_1<w_2$. 

Second, in Section~\ref{sec:minimality}, we show that monotone sets are perimeter-minimizing and conclude that $S_\alpha^+$ is perimeter-minimizing when $\frac{\alpha(w_2)-\alpha(w_1)}{w_2-w_1} \ge -1$ for all $w_1<w_2$. This is a consequence of the kinematic formula, which expresses the perimeter of a subset $E\subset \HH$ in terms of an integral over the set of horizontal lines that intersect $\partial E$. When $E$ is monotone, then almost every horizontal line intersects $\partial E$ in at most one point. In this case, deformations of $E$ increase the number of intersections and thus increase perimeter. Proposition~\ref{prop:strips} follows immediately from the results of Section~\ref{sec:deforming-ruled} and Section~\ref{sec:minimality}.

\subsection{Graphical strips}\label{sec:strip-prelims}
In this section, we prove Lemma~\ref{lem:strip-injectivity} and Lemma~\ref{lem:alternate}. 

\begin{proof}[Proof of Lemma~\ref{lem:strip-injectivity}]
  By construction, every point in $S_\alpha$ is contained in a ruling and every ruling intersects the $z$--axis. In order for $S_\alpha$ to be a graphical strip, we need it to be an intrinsic graph over $K$.

  Let $\Theta\from [-1,1]\times \R\to \HH$, $\Theta(x,w)=(x,x\alpha(w),\eta(w))$ parametrize $S_\alpha$ and let $p\from [-1,1]\times \R\to K$ be the map 
  \begin{multline*}
    p(x,w) =\Pi(\Theta(x,w))
    =\Pi \left(x, x \alpha(w), \eta(w)\right) \\
    =\left(x, 0, \eta(w) - \frac{\alpha(w)}{2} x^2\right) = \left(x, 0, \eta(w)(1-x^2) + w x^2\right).
  \end{multline*}
    
  Suppose that $\eta$ is surjective and $\frac{\alpha(s)-\alpha(t)}{s-t} \ge -2$ for all $s<t$.  We claim that $p|_{([-1,1]\setminus \{0\})\times \R}$ is a homeomorphism. Let $h_x(w):= z(p(x,w)) = \eta(w)(1-x^2) + w x^2$. Since $\eta$ is nondecreasing, $h_x$ is an increasing surjective function when $x\ne 0$. It follows that $p|_{([-1,1]\setminus \{0\}) \times \R}$ is a bijection. For any $a$ and $b$ such that $-1\le a<b<0$ or $0<a<b\le 1$ and any $c<d$,
  $$p((a,b)\times(c,d)) = \{(x,0,z)\mid x\in (a,b), z \in (h_x(c),h_x(d))\}$$
  is an open set, so $p|_{([-1,1]\setminus \{0\}) \times \R}$ is an open map. Therefore, $p$ restricts to a homeomorphism from $([-1,1]\setminus \{0\}) \times \R$ to $K\setminus \langle Z\rangle$.

  Let $\mathbf{q}=(q_1,q_2):=(p|_{([-1,1]\setminus \{0\}) \times \R})^{-1}\from K\setminus\langle Z\rangle \to ([-1,1]\setminus \{0\}) \times \R$.
  For $(x,z)\in K$, let
  $$f_\alpha(x,z) :=\begin{cases}
    x \alpha(q_2(x,z)) & x\ne 0\\
    0 & x = 0.
  \end{cases}.$$
  Then $f_\alpha(p(x,w)) = x \alpha(w)$, so
  $$\Psi_{f_\alpha}(p(x,w)) = \Pi(\Theta(x,w)) Y^{x\alpha(w)} = \Theta(x,w),$$
  and $\Gamma_{f_\alpha}=S_\alpha$.

  We claim that $f_\alpha$ is continuous. Since $\mathbf{q}$ is continuous, $f_\alpha$ is continuous on $K\setminus \langle Z \rangle$. Let $z\in \R$ and let $w_1<w_2$ be such that $\eta(w_1)< z<\eta(w_2)$ and let $C=\max_{u\in [w_1,w_2]} |\alpha(u)|$. Let
  $$D=p([-1,1]\times [w_1,w_2]) = \{(x,0,z)\mid x\in [-1,1], z \in [h_x(c),h_x(d)]\};$$
  this contains a neighborhood of $(0,0,z)$, and if $v=(x,0,z)\in D$, then $|f_\alpha(v)|\le C x$. It follows that $f_\alpha$ is continuous at $(0,0,z)$ for any $z\in \R$. 
  
  Conversely, if there are $s<t$ such that $\frac{\alpha(s)-\alpha(t)}{s-t} < - 2$, then $\eta(t)<\eta(s)$, and there is some $x>0$ such that $h_x(s)=h_x(t)$. Then $p(x,s)=p(x,t)$, but $y(\Theta(x,s))=x\alpha(s)\ne y(\Theta(x,t))$. Therefore, $S_\alpha$ is not an intrinsic graph. Likewise, if $\eta$ is not surjective, then $\Pi(S_\alpha)$ is a strict subset of $K$.
\end{proof}

\begin{proof}[Proof of Lemma~\ref{lem:alternate}]
  Suppose that $\frac{\alpha(w_2)-\alpha(w_1)}{w_2-w_1} \ge -1$ for all $s<t$. Let $\eta(w):= w+\frac{\alpha(w)}{2}$; this is invertible and surjective by our hypothesis on $\alpha$, and we define $\sigma(z)=\alpha(\eta^{-1}(z))$. When $z=\eta(w)$,
  $$[(-1, -\alpha(w), \eta(w)), (1, \alpha(w), \eta(w))] = [(-1, -\sigma(z), z), (1, \sigma(z), z)],$$
  so
  $$S_\alpha = \bigcup_{z\in \R} [(-1, -\sigma(z), z), (1, \sigma(z), z)].$$
  
  Let $z_1<z_2$. We claim that $-2 \le \frac{\sigma(z_2) - \sigma(z_1)}{z_2 - z_1} < 2$. Let $w_i=\eta^{-1}(z_i)$ and $a_i=\alpha(w_i)=\sigma(z_i)$. Then $z_i=\eta(w_i)=w_i+\frac{a_i}{2}$. Since $\eta$ is monotone, we have $w_1<w_2$, so
  $$w_2 - w_1 = z_2 - z_1 - \frac{a_2-a_1}{2} > 0,$$
  i.e., $a_2 - a_1 < 2(z_2 - z_1)$.
  Furthermore, 
  $$\frac{a_2 - a_1}{w_2-w_1} = \frac{\alpha(w_2)-\alpha(w_1)}{w_2-w_1} \ge -1,$$
  so
  $$a_2 - a_1 \ge w_1-w_2 = z_1 - z_2 + \frac{1}{2}(a_2-a_1),$$
  so $a_2 - a_1 \ge -2(z_2-z_1)$. 
  Therefore, $-2 \le \frac{\sigma(z_2) - \sigma(z_1)}{z_2 - z_1} < 2$.

  Conversely, suppose that 
  \begin{equation}\label{eq:salpha-as-sigma}
    S_\alpha = \bigcup_{z\in \R} [(-1, -\sigma(z), z), (1, \sigma(z), z)]
  \end{equation}
  for some $\sigma$ such that
  \begin{equation}\label{eq:sigma-cond}
    -2 \le \frac{\sigma(z_2) - \sigma(z_1)}{z_2 - z_1} < 2
  \end{equation}
  for all $z_1<z_2$. We claim that $\frac{\alpha(w_2)-\alpha(w_1)}{w_2-w_1} \ge -1$ for all $w_1<w_2$.

  We have
  $$\{(x,y,z) \in S_\alpha \mid x=1\} = \{(1,\sigma(z),z)\mid z\in \R\} = \{(1,\alpha(w),\eta(w))\mid w\in \R\},$$
  so $\sigma(\eta(w))=\alpha(w)$ for all $w\in \R$. Let $w_1<w_2$. Since $\eta$ is nondecreasing, we have $\eta(w_1)\le \eta(w_2)$. In fact $\eta(w_1)<\eta(w_2)$; otherwise, $S_\alpha$ would have two rulings with the same $z$--coordinate.

  By \eqref{eq:sigma-cond} with $z_1=\eta(w_1)$, $z_2=\eta(w_2)$, 
  \begin{align*}
    \sigma\left(\eta(w_2)\right) - \sigma\left(\eta(w_1)\right) & \ge -2\left(\eta(w_2) - \eta(w_1)\right)\\
    \alpha(w_2) - \alpha(w_1) & \ge -2 \left(w_2 + \frac{\alpha(w_2)}{2} - w_1 - \frac{\alpha(w_1)}{2}\right)\\
    2(\alpha(w_2) - \alpha(w_1))&\ge -2(w_2-w_1),
  \end{align*}
  so $\frac{\alpha(w_2)-\alpha(w_1)}{w_2-w_1} \ge -1$, as desired.
\end{proof}

\subsection{Deforming ruled surfaces}\label{sec:deforming-ruled}
Recall that for any function $\tau\from \R\to \R$, we define $\gamma_\tau:=\{(0,y,z)\mid y=\tau(z)\}$ to be the graph of $\tau$ in the $yz$--plane, so that $S_\alpha$ has boundary
$$\partial S_\alpha = X^{-1} \gamma_\alpha \cup X\gamma_{-\alpha}.$$
In this section, for any sufficiently small $\tau$, we construct a ruled graph $\Sigma_{\tau,\alpha}$ over $K^+=[0,1] \times \{0\} \times \R$ with boundary $\partial \Sigma_{\tau,\alpha}= \gamma_\tau \cup X \gamma_\alpha$. By gluing two such graphs together, we obtain a surface
$$S_{\alpha,\tau} := \Sigma_{\tau,\alpha} \cup s_{-1,-1}(\Sigma_{-\tau,\alpha})$$
which is a deformation of $S_\alpha$.

First, we construct $\Sigma_{\tau,\alpha}$.
\begin{lemma}\label{lem:ruled-existence}
  Let $\alpha,\tau\from \R\to \R$ be continuous functions. Let $\eta(w):=w+\frac{\alpha(w)}{2}$, and suppose that $\frac{\alpha(s)-\alpha(t)}{s-t} \ge -2$ for all $s<t$ and $\eta(\R)=\R$. Suppose that $\Lip(\tau)<2$.
  Then there is a unique ruled graph $\Sigma_{\tau,\alpha}$ with boundary $\partial \Sigma_{\tau,\alpha}=\gamma_\tau\cup X\gamma_\alpha$.
\end{lemma}
\begin{proof}
  We first construct a tilted coordinate system on the $yz$--plane. Let $P\subset \HH$ be the $yz$--plane and let $Q = Y + \frac{Z}{2}$. Then $P=\langle Q, Z\rangle$. Since $\Lip(\tau)<2$, we can write $\gamma_\tau$ as a $P$--graph. That is, any point $Z^z Y^{\tau(z)}\in \gamma_\tau$ can be written
  \begin{equation}\label{eq:z-tau-zeta-tau}
    Z^z Y^{\tau(z)} = Z^{z - \frac{\tau(z)}{2}} P^{\tau(z)} = Z^{\zeta(z)} P^{\tau(z)}
  \end{equation}
  where $\zeta(z) :=z - \frac{\tau(z)}{2}$. Then $\zeta$ is invertible, so, letting $w=\zeta(z)$,
  $$\gamma_\tau = \{Z^{w} P^{\tau(\zeta^{-1}(w))} \mid w\in \R\}.$$
  Let $\beta(w) := \tau(\zeta^{-1}(w))$.

  Let $w\in\R$ and let $p=X Z^w Y^{\alpha(w)} \in X \gamma_\alpha$. For any $m \in \R$,
  \begin{multline}\label{eq:horiz-rule-sigma}
    p \cdot (X+ m Y)^{-1} = (1, \alpha(w),\eta(w)) (-1, -m, 0) \\ = \left(0, \alpha(w) - m, \eta(w) + \frac{\alpha(w) - m}{2}\right) = Z^{\eta(w)} P^{\alpha(w) - m},
  \end{multline}
  which lies in $\gamma_\tau$ if and only if $m = \alpha(w) - \beta(\eta(w))$.
  In this case,
  $$Z^{\eta(w)} P^{\alpha(w) - m}= Z^{\eta(w)} P^{\beta(\eta(w))} \stackrel{\eqref{eq:z-tau-zeta-tau}}{=} Z^{\zeta^{-1}(\eta(w))} Y^{\tau(\zeta^{-1}(\eta(w)))}= Z^{h(w)} Y^{\tau(h(w))},$$
  where $h(w) := \zeta^{-1}(\eta(w))$. Let
  \begin{equation}\label{eq:rulings-of-tau-alpha}
    R_w:= \left[X Z^w Y^{\alpha(w)}, Z^{\eta(w)} P^{\beta(\eta(w))}\right]
    = \left[X Z^w Y^{\alpha(w)}, Z^{h(w)} Y^{\tau(h(w))}\right]
  \end{equation}
%    \\ 
%        &= \left[(-1,\alpha(w),\eta(w)), \left(0, \tau(\zeta^{-1}(\eta(w))), \zeta^{-1}(\eta(w))\right) \right]
  be the segment connecting $p$ to $\gamma_\tau$. Since $\eta$ is surjective, every point of $\gamma_\tau$ is the endpoint of some such segment.
  We call the union of these segments
  $$\Sigma_{\tau,\alpha} := \bigcup_w R_w.$$
  This is a surface ruled by the $R_w$ and bounded by $\gamma_\tau$ and $X \gamma_\alpha$.

  Let $\delta(w)=\alpha(w)-\beta(\eta(w))$, let
  $$\lambda_w(x) := Z^{h(w)} Y^{\tau(h(w))} (X + \delta(w) Y)^{x} = X Z^w Y^{\alpha(w)} (X + \delta(w) Y)^{x-1}$$
  parametrize $R_w$, and let $\Theta(x,w) := \lambda_w(x)$ parametrize $\Sigma_{\tau,\alpha}$. We claim that $\Sigma_{\tau,\alpha}$ is an intrinsic graph.

  For each $w$, there is some quadratic $g_w$ such that  $\Pi(\lambda_w(x))=(x,0,g_w(x))$. We have
  $$g_w(0)=z(\Pi(\lambda_w(0))) = z(\Pi(Z^{h(w)} Y^{\tau(h(w))})) = h(w),$$
  $$g_w(1) = z(\Pi(\lambda_w(1))) = z(\Pi(X Z^w Y^{\alpha(w)})) = w,$$
  and $g''_w(x)= -\slope R_w = -\delta(w)$, so
  $$g_w(x) = (1 - x) h(w) + x w + \frac{\delta(w)}{2} x(1-x).$$
  Then $\Pi(\Theta(x,w))= (x, 0, g_w(x))$. 
  We claim that for every $x\in (0,1]$, the map $w\mapsto g_w(x)$ is monotone increasing.

  Suppose that there are $x_0\in (0,1]$ and $w_1<w_2$ such that $g_{w_1}(x_0)=g_{w_2}(x_0)$. Let $p(x)=g_{w_2}(x)-g_{w_1}(x)$. Then 
  $p(0)=h(w_2)-h(w_1)\ge 0$, $p(1)=w_2-w_1> 0$, and $p''(x)=2c$ for all $x$, where $c=\frac{\delta(w_1)-\delta(w_2)}{2}$. Since $p$ is quadratic, $p(0)\ge 0$, $p(1) > 0$, and $p(x_0)= 0$, it has a minimum at some point $r \in (0,1)$ and $p(r)\le 0$. Therefore, $p(x)=p(r) + c (x-r)^2$ and $c>0$.

  This implies
  $$h(w_2) - h(w_1) =  p(1) = p(r) + c (1-r)^2.$$
  Since $h$ is nondecreasing, $p(r)<0$, and $r\in (0,1)$,
  $$0\le h(w_2) - h(w_1) \le  c (1-r)^2.$$

  By equation \eqref{eq:point-slope}, since $\lambda_w$ is a line of slope $\delta(w)$ through $(0,\tau(h(w)),h(w))$, we can also write
  $$g_w(x) = h(w) - \tau(h(w)) x - \frac{\delta(w)}{2}x^2.$$
  Then $g_w'(0) = -\tau(h(w))$, and 
  $$|\tau(h(w_2)) - \tau(h(w_1))| = |p'(0)| = 2 c |1-r| > 2 c(1-r)^2 \ge 2 |h(w_2) - h(w_1)|,$$
  which contradicts the assumption that $\Lip(\tau)<2$.

  Thus, for every $x\in (0, 1]$, the map $w\mapsto g_w(x)$ is increasing, and the map $w\mapsto g_w(0)=h(w)$ is non-decreasing. 
  Suppose that $p,q\in \Sigma_{\tau,\alpha}$ satisfy $\Pi(p)=\Pi(q)$. There are $x_1,x_2\in [0,1]$ and $w_1,w_2\in \R$ such that $p=\Theta(x_1,w_1)$ and $q=\Theta(x_2,w_2)$; since $\Pi(p)=\Pi(q)$, we have $x_1=x_2$. By the above, if $x_1\in (0,1]$, then $p=q$. Otherwise, if $x_1=x_2=0$, then $p$ and $q$ lie on $\gamma_\tau$, but $\Pi$ is injective on $\gamma_\tau$, so $p=q$. Therefore, $\Pi$ is injective on $\Sigma_{\tau,\alpha}$, so $\Sigma_{\tau,\alpha}$ is an intrinsic graph.
\end{proof}

Next, we calculate the area of $\Sigma_{\tau,\alpha}$ using the area formula \eqref{eq:area-formula}.
\begin{lemma}\label{lem:ruled-area-formula}
  Suppose that $\alpha$ is Lipschitz and that $S_\alpha$ is a graphical strip over $K$. 
  Let $\eta(w) = w + \frac{\alpha(w)}{2}$, $\zeta(z) =z - \frac{\tau(z)}{2}$, $\beta(w) = \tau(\zeta^{-1}(w))$, and $\delta(w)=\alpha(w) - \beta(\eta(w))$ as above. By Lemma~\ref{lem:strip-injectivity}, $\alpha'(w)\ge -2$ for almost every $w$ and $\eta$ is nondecreasing.
  
  Let
  $$\Theta(x,w) := X Z^w Y^{\alpha(w)} (X + \delta(w) Y)^{x-1}$$
  parametrize $\Sigma_{\tau,\alpha}$.
  Let $a<b$. Then
  \begin{equation}\label{eq:ruled-area}
    \area \Theta([0,1]\times[a,b]) = \int_{a}^b \sqrt{1+\delta(w)^2}\left(1+\frac{\alpha'(w)}{2}\right) \ud w - \frac{1}{6} \int_{\delta(a)}^{\delta(b)} \sqrt{1+m^2}\ud m.
  \end{equation}
\end{lemma}
\begin{proof}
  By the area formula,
  \begin{multline*}
    \area \Theta([0,1]\times[a,b]) = \int_a^b \int_{0}^1 \sqrt{1+\delta(w)^2}\cdot \big|J[\Pi\circ \Theta](x,w)\big|\ud x \ud w\\
    =\int_a^b \sqrt{1+\delta(w)^2} j(w)\ud w,
  \end{multline*}
  where $j(w)=\int_0^1 \big|J[\Pi\circ \Theta](x,w)\big|\ud x.$

  Each ruling $\Theta([0,1]\times \{w\})$ is a line with slope $\delta(w)$ going through $X Z^w Y^{\alpha(w)}$. Since $\Pi(X Z^w Y^{\alpha(w)})=(1,0,w)$, by \eqref{eq:point-slope},
  \begin{align*} 
    \Pi(\Theta(x,w))
%     = \Pi(X Z^w Y^{\alpha(w)}) + \left(x-1, 0, - \alpha(w) (x-1) - \frac{\delta(w)}{2}(x-1)^2\right) 
     = \left(x, 0, w - \alpha(w) (x-1) - \frac{\delta(w)}{2}(x-1)^2\right)
  \end{align*}
  Since $\Pi\circ \Theta$ is injective, for any $x\in (0,1)$, $w\mapsto z(\Pi(\Theta(x,w)))$ must be increasing. Then
  $$J[\Pi\circ \Theta](x,w)%=\partial_w\left[w + (1-x) \alpha(w) - \frac{(1-x)^2}{2} \delta(w) \right](x,w)
  = 1 + (1-x) \alpha'(w) - \frac{(1-x)^2}{2}\delta'(w) \ge 0,$$
  and
  $$j(w) = \int_0^1 1 + (1-x) \alpha'(w) - \frac{(1-x)^2}{2} \delta'(w) \ud x = 1+\frac{\alpha'(w)}{2}- \frac{\delta'(w)}{6}.$$

  Therefore,
  \begin{multline*}
    \area \Theta([0,1]\times[a,b]) = \int_a^b \sqrt{1+\delta(w)^2}\left(1 + \frac{\alpha'(w)}{2} - \frac{\delta'(w)}{6}\right) \ud w \\
    = \int_a^b \int_{0}^1 \sqrt{1+\delta(w)^2} \left(1 + \frac{\alpha'(w)}{2}\right) \ud w - \int_{\delta(a)}^{\delta(b)} \sqrt{1+m^2} \ud m,
  \end{multline*}
  where we substitute $m=\delta(w)$ in the last equality.
\end{proof}

For $\tau$ such that $\Lip(\tau)<2$, let
$$S_{\alpha,\tau} := \Sigma_{\tau,\alpha} \cup s_{-1,-1}(\Sigma_{-\tau,\alpha}).$$
This is a deformation of $S_\alpha$, and we will use Lemma~\ref{lem:ruled-area-formula} to calculate the second variation of its area.
\begin{proof}[Proof of Proposition~\ref{prop:vf-strips}]
  Let $\lambda$ be small enough that $\Lip(\lambda\tau)<2$, so that $\Sigma_{\lambda \tau, \alpha}$ exists. Let $\zeta_\lambda(z) = z - \frac{\lambda \tau(z)}{2}$, let
  $$\beta_\lambda(w) = \lambda \tau(\zeta^{-1}_\lambda(w)),$$
  and let $\delta_\lambda(w) = \alpha(w) - \beta_\lambda(\eta(w))$, so that
  $$\Theta_{\lambda \tau, \alpha}(x,w) = 
  X Z^w Y^{\alpha(w)} (X + \delta_\lambda(w) Y)^{x-1}$$
  parametrizes $\Sigma_{\lambda \tau, \alpha}$ and $s_{-1,-1}\circ \Theta_{-\lambda \tau, \alpha}$ parametrizes $s_{-1,-1}(\Sigma_{-\lambda \tau,\alpha})$.
  Let
  $$R_{w,\lambda}:=\Theta_{\lambda \tau, \alpha}([0,1]\times\{w\})=\left[Z^{\zeta_\lambda^{-1}(\eta(w))} Y^{\lambda \tau(\zeta_\lambda^{-1}(\eta(w)))} ,X Z^w Y^{\alpha(w)}\right].$$

  Let $w_1<w_2$ be such that $\eta^{-1}(\eta(w_i))=\{w_i\}$ and let $D=\{(x,y,z) \in \HH\mid \eta(w _1)\le z\le \eta(w _2) \}$. Then for $z\not \in (\eta(w_1),\eta(w_2))$, we have $\zeta_\lambda(z) = z$, so $\zeta_\lambda^{-1}(z) = z$ and
  $$\zeta_\lambda^{-1}((\eta(w_1),\eta(w_2)))\subset (\eta(w_1),\eta(w_2)).$$
  We claim that for any $|\lambda|<1$,
  \begin{equation}\label{eq:int-limits-Salpha}
    S_{\alpha,\lambda\tau}\cap D = \Theta_{\lambda \tau,\alpha}([0,1]\times [w_1,w_2]) \cup s_{-1,-1}(\Theta_{-\lambda \tau,\alpha}([0,1]\times [w_1,w_2])).
  \end{equation}

  It suffices to show that $R_{w,\lambda}\subset D$ when $w\in [w_1,w_2]$ and $R_{w,\lambda}\cap D=\emptyset$ otherwise.
  The endpoints of $R_{w,\lambda}$ have $z$--coordinates
  $$z\left(Z^{\zeta_\lambda^{-1}(\eta(w))} Y^{\lambda \tau(\zeta_\lambda^{-1}(\eta(w)))}\right) = \zeta_\lambda^{-1}(\eta(w))$$
  and $z(X Z^w Y^{\alpha(w)}) = \eta(w)$. When $w\in [w_1,w_2]$, we have $\zeta_\lambda^{-1}(\eta(w)) \in [\eta(w_1),\eta(w_2)]$ and $\eta(w) \in [\eta(w_1),\eta(w_2)]$, so $R_{w,\lambda}\subset D$. When $w\not\in [w_1,w_2]$, we have $\zeta_\lambda^{-1}(\eta(w))=\eta(w) \not \in [\eta(w_1),\eta(w_2)]$, so $R_{w,\lambda}$ is disjoint from $D$.

  On the other hand, when $w \in [w_1,w_2]$, we have $\eta(w) \in [\eta(w_1),\eta(w_2)]$ and $\zeta_\lambda^{-1}(\eta(w))\in [\eta(w_1),\eta(w_2)]$, so both endpoints of $\Theta_{\lambda \tau, \alpha}([0,1]\times\{w\})$ lie in $D$. It follows that $\Theta_{\lambda \tau, \alpha}(x,w)\in D$ if and only if $w\in [w_1,w_2]$, which shows \eqref{eq:int-limits-Salpha}.

  Therefore, by Lemma~\ref{lem:ruled-area-formula},
  \begin{multline*}
    \area (S_{\alpha,\lambda \tau}\cap D) 
    =\area \Theta_{\lambda \tau, \alpha}([0,1]\times [w_1,w_2]) + \area \Theta_{-\lambda \tau, \alpha}([0,1]\times [w_1,w_2]) \\ 
    =\int_{w_1}^{w_2}  \left(\sqrt{1+\delta_\lambda(w)^2}+\sqrt{1+\delta_{-\lambda}(w)^2}\right)\left(1+\frac{\alpha'(w)}{2}\right) \ud w - \frac{1}{3} \int_{\delta(w_1)}^{\delta(w_2)} \sqrt{1+m^2}\ud m.
  \end{multline*}
  
  It remains to differentiate $\delta_\lambda(w)$. Let $z_w(\lambda):=\zeta^{-1}_{\lambda}(w)$.
  Then
  $$\zeta_\lambda(z_w(\lambda)) = z_w(\lambda) - \frac{\lambda}{2} \tau(z_w(\lambda)) = w.$$
  We differentiate both sides with respect to $\lambda$ to find
  $$z_w'(\lambda) - \frac{1}{2} \tau(z_w(\lambda)) + \frac{\lambda}{2} \tau'(z_w(\lambda)) z_w'(\lambda) = 0.$$
%and
%$$z_w''(\lambda) + \frac{1}{2} \tau'(z_w(\lambda))g'(\lambda) + \frac{1 }{2} \tau'(z_w(\lambda))g'_t(\lambda) + \frac{\lambda}{2} \tau''(z_w(\lambda)) g'_t(\lambda)^2 + \frac{\lambda}{2} \tau'(z_w(\lambda)) g''_t(\lambda)= 0.$$
  Setting $\lambda=0$ and noting that $z_w(0)=\zeta^{-1}_{0}(w)=w$, we find
  $z_w'(0)= \frac{1}{2} \tau(w)$. % and $z_w''(\lambda) = -\tau'(t)g'_t(0)=\frac{1}{2} \tau'(t)\tau(t)$.
  Likewise, let $b_w(\lambda):=\beta_{\lambda}(w)=\lambda \tau(z_w(\lambda))$. Then
  $$b_w'(\lambda)= \tau(z_w(\lambda)) + \lambda \tau'(z_w(\lambda)) z_w'(\lambda),$$
  so $b_w'(0)=\tau(t)$ and $b_w''(0)=2 \tau'(z_w(0)) z_w'(0)= \tau'(w)\tau(w)$, i.e.,
  $$\beta_{\lambda}(w) = \lambda \tau(w) + \frac{\lambda^2}{2} \tau'(w)\tau(w) + O_\tau(\lambda^3)$$
  for sufficiently small $\lambda$, where $O_\tau(\lambda^3)$ denotes an error term of magnitude at most $C(\tau) \lambda^3$.

  Substituting this into the power series
  $$\sqrt{1+(b-a)^2}=\sqrt{1+a^2} - \frac{a}{\sqrt{1+a^2}} b + \frac{1}{2(1+a^2)^{\frac{3}{2}}} b^2 + O(b^3),$$
  and abbreviating $\alpha=\alpha(w)$, $z=\eta(w)$, $\tau=\tau(z)$, $\frac{\ud \tau}{\ud z} = \tau'(\eta(w))$, we find
  \begin{align*}
    &\sqrt{1+(\beta_{\lambda}(\eta(w))-\alpha(w))^2} + \sqrt{1+(\beta_{-\lambda }(\eta(w))-\alpha(w))^2} \\
    &\qquad = \sqrt{1+\alpha^2} - \frac{\alpha}{\sqrt{1+\alpha^2}} \left(\lambda \tau + \frac{\lambda^2}{2} \frac{\ud \tau}{\ud z} \tau \right) + \frac{(\lambda \tau)^2}{2 (1+\alpha^2)^{\frac{3}{2}}}\\
    &\qquad \quad + \sqrt{1+\alpha^2} - \frac{\alpha}{\sqrt{1+\alpha^2}} \left(-\lambda \tau + \frac{\lambda^2}{2} \frac{\ud \tau}{\ud z} \tau \right) + \frac{(\lambda \tau)^2}{2 (1+\alpha^2)^{\frac{3}{2}}} + O_\tau(\lambda^3) \\
    &\qquad= 2 \sqrt{1+\alpha^2} - \frac{\alpha}{\sqrt{1+\alpha^2}} \frac{\ud\tau}{\ud z}\tau \lambda^2 + \frac{\tau^2}{(1+\alpha^2)^{\frac{3}{2}}}\lambda^2 + O_\tau(\lambda^3).
  \end{align*}
  Substituting this into \eqref{eq:ruled-area}, using the fact that $\frac{\ud z}{\ud w} = \frac{\ud \eta}{\ud w} = 1+\frac{1}{2} \frac{\ud \alpha}{\ud w}$, and integrating by parts we find
  \begin{align*}
    \area & (S_{\alpha,\lambda \tau}\cap D) - \area(S_{\alpha,0}\cap D)\\
          & =
            \lambda^2 \int_{w_1}^{w_2} \left(-\frac{\alpha}{\sqrt{1+\alpha^2}} \tau\frac{\ud\tau}{\ud z} + \frac{\tau^2}{(1+\alpha^2)^{\frac{3}{2}}}\right) \left(1+\frac{1}{2} \frac{\ud \alpha}{\ud w}\right) \ud w + O_\tau(\lambda^3 |w_2-w_1|) \\
          & =
            \lambda^2 \int_{w_1}^{w_2} -\frac{\alpha}{\sqrt{1+\alpha^2}} \cdot \tau \frac{\ud\tau}{\ud z} \frac{\ud z}{\ud w} + \frac{\tau^2}{(1+\alpha^2)^{\frac{3}{2}}}\left(1+ \frac{1}{2} \frac{\ud \alpha}{\ud w}\right) \ud w + O_\tau(\lambda^3 |w_2-w_1|) \\
          & =\lambda^2\int_{w_1}^{w_2} \frac{1}{(1+\alpha^2)^{\frac{3}{2}}}\frac{\ud \alpha}{\ud w}\cdot  \frac{\tau^2}{2} + \frac{\tau^2}{(1+\alpha^2)^{\frac{3}{2}}} \left(1+ \frac{1}{2} \frac{\ud \alpha}{\ud w}\right) \ud w + O_\tau(\lambda^3 |w_2-w_1|)\\
          & =\lambda^2 \int_{w_1}^{w_2} \frac{\tau^2}{(1+\alpha^2)^{\frac{3}{2}}}\left(1+ \frac{\ud \alpha}{\ud w}\right) \ud w + O_\tau(\lambda^3 |w_2-w_1|).
  \end{align*}
  This proves the proposition.
\end{proof}

Though the proposition deals with the case that $\alpha$ is Lipschitz, we can use it to find a necessary condition for $S_\alpha$ to be perimeter-minimizing in the general case.

\begin{lemma}\label{lem:min-implies-bound}
  Suppose that $S_\alpha$ is a graphical strip over $K$ and $S_\alpha$ is area-minimizing in $U$. Then $\frac{\alpha(w_2)-\alpha(w_1)}{w_2-w_1} \ge -1$ for all $w_1<w_2$.
\end{lemma}
\begin{proof}
  Note that $S_\alpha^+ \symdiff S_{\alpha,\lambda\tau}^+$ is not compactly contained in $U$, so $S_{\alpha,\lambda \tau}$ is not a valid competitor for $S_\alpha$ in $U$. We fix this by scaling $S_\alpha$ slightly. This also lets us replace $\alpha$ by a Lipschitz function.

  Let $S_\alpha$ be a graphical strip over $K$ and suppose that $S_\alpha$ is perimeter-minimizing in $U$. Let $\epsilon>0$ and let 
  $\hat{S}:=s_{1+\epsilon,1+\epsilon}(S_\alpha).$
  Then $\hat{S}$ is a graphical strip over $s_{1+\epsilon,1+\epsilon}(K)$, so $\hat{S} \cap U$ is a graphical strip over $S$. 
  Let $f\from K\to \R$ be the function such that $S_\alpha=\Gamma_f$. Then $\hat{S}=\Gamma_{\hat{f}}$, where $\hat{f}\from s_{1+\epsilon,1+\epsilon}(K)\to \R$ is the function
  $$\hat{f}_\epsilon(p)=(1+\epsilon) f\left(s^{-1}_{1+\epsilon,1+\epsilon}(p)\right).$$
  Let $\hat{\alpha}_\epsilon(w) = \hat{\alpha}(w):=\hat{f}_\epsilon(1,w)$. Then   $\hat{S} \cap U= S_{\hat{\alpha}}$, and $\hat{S}$ is perimeter-minimizing in $\hat{U}:=s_{1+\epsilon,1+\epsilon}(U)$.

  We claim that $\hat{\alpha}$ is Lipschitz.
  For $w\in \R$, let $p_w=X Z^w Y^{\hat{\alpha}(w)}$. Then $p_w\in \hat{S}$ and there is a ruling $\hat{R}_w$ through $p_w$ with slope $\hat{\alpha}(w)$. By \eqref{eq:point-slope},
  $$\Pi(\hat{R}_w)=\{(x,0,g_w(x))\mid x\in [-1-\epsilon, 1+\epsilon]\},$$
  where 
  $$g_w(x) = w - \hat{\alpha}(w) (x-1) - \frac{\hat{\alpha}(w)}{2} (x-1)^2.$$

  Let $s<t$, so that $g_s(1) = s < t = g_t(1)$. By Lemma~\ref{lem:ordering-rulings}, $g_s(x) \le g_t(x)$ for all $x\in [-1-\epsilon, 1+\epsilon]$. Then
  $$g_t(0)-g_s(0) = t + \frac{\hat{\alpha}(t)}{2} - s - \frac{\hat{\alpha}(s)}{2} \ge 0,$$
  so $\hat{\alpha}(t) - \hat{\alpha}(s) \ge -2 (t-s)$ and
  $$g_t(1+\epsilon) -g_s(1+\epsilon) = t - s - \left(\epsilon + \frac{\epsilon^2}{2}\right)(\hat{\alpha}(t)-\hat{\alpha}(s)) \ge 0,$$
  i.e,
  $$\hat{\alpha}(t) - \hat{\alpha}(s) \le \left(\epsilon + \frac{\epsilon^2}{2}\right)^{-1} (t-s).$$
  That is, $\hat{\alpha}$ is Lipschitz.
  
  We claim that $\hat{\alpha}'(w)\ge -1$ for almost every $w\in \R$. 
  Suppose by contradiction that $\{w\mid \hat{\alpha}'(w)<-1\}$ has positive measure.
  Let
  $$\II(\tau) = \int_{-\infty}^\infty \frac{\tau(w)^2}{(1+\hat{\alpha}(w)^2)^{\frac{3}{2}}}\left(1+ \hat{\alpha}'(w)\right) \ud w;$$
  we claim that there is a $\tau\in C^\infty_c$ such that $\II(\tau)<0$.
  
  Let $\hat{\eta}(w):=w + \frac{\hat{\alpha}(w)}{2}$.
  We consider two cases, depending on whether $\hat{\eta}$ is injective. 
  If $\hat{\eta}$ is not injective, then there is some $z_0\in \R$ such that $\hat{\eta}^{-1}(z_0)=[a,b]$ for some $a<b$. For any $w\in (a,b)$, we have $\hat{\eta}(w)=w+\frac{\hat{\alpha}(w)}{2} = z_0$, so $\hat{\alpha}(w)=2z_0-2w$ and thus $\hat{\alpha}'(w)=-2$. Therefore,
  $$\II(\one_{\{z_0\}}) = \int_{-\infty}^{\infty} \frac{\one_{\{z_0\}}(\hat{\eta}(w))} {(1+\hat{\alpha}(w)^2)^{\frac{3}{2}}} \left(1 + \hat{\alpha}'(w) \right) \ud w 
  = \int_{a}^{b} - \frac{1} {(1+\hat{\alpha}(w)^2)^{\frac{3}{2}}} \ud w<0.$$
  
  If $\hat{\eta}$ is injective, we let $w_0\in \R$ be a Lebesgue point of $\hat{\alpha}'$ such that $\hat{\alpha}'(w_0)<-1$ and let $I_r=(\hat{\eta}(w_0-r), \hat{\eta}(w_0+r))$. Since $\hat{\alpha}$ is continuous, $w_0$ is also a Lebesgue point of $(1+\hat{\alpha}'(w))(1+\hat{\alpha}(w)^2)^{-\frac{3}{2}}$, and
  $$\lim_{r\to 0} \frac{\II(\one_{I_r})}{2r}  = \lim_{r\to 0} \frac{1}{2r} \int_{w_0-r}^{w_0+r} \frac{1} {(1+\hat{\alpha}(w)^2)^{\frac{3}{2}}} \left(1+\hat{\alpha}'(w) \right) \ud w = \frac{1+\hat{\alpha}'(w_0)}{(1+\hat{\alpha}(w_0)^2)^{\frac{3}{2}}}<0.$$
  
  In either case, there is some bounded interval $I$ (possibly $I=\{z_0\}$) such that $\II(\one_I)<0$. Let $\tau_i \in C^\infty_c$ be a sequence of uniformly bounded functions with uniformly bounded supports such that $\tau_i\to \one_I$ pointwise. Then $\lim_{i\to \infty} \II(\tau_i) =\II(\one_I)$ by dominated convergence, so there is a $\tau_i$ such that $\II(\tau_i)<0$. Let $F_\lambda = (\hat{S} \setminus U) \cup S_{\hat{\alpha},\lambda \tau_i}.$
  Then $F_\lambda$ is a deformation of $F_0=\hat{S}$, and there is an $r>0$ such that 
  $F_\lambda^+ \symdiff \hat{S}^+\Subset \hat{U}\cap B(\0,r).$
  By Proposition~\ref{prop:vf-strips}, there is a $C>0$ such that 
  $$\area(B(\0,r) \cap F_\lambda) \le \area(B(\0,r) \cap \hat{S}) + \lambda^2 \II(\tau_i) + C \lambda^3,$$
  so $\area(B(\0,r) \cap F_\lambda) < \area(B(\0,r) \cap \hat{S})$ when $\lambda$ is sufficiently small. This contradicts the fact that $\hat{S}$ is area-minimizing on $\hat{U}$, so $\hat{\alpha}'(w)\ge -1$ for almost every $w\in \R$.

  Thus, if $S_\alpha$ is area-minimizing in $U$, then $\frac{\hat{\alpha}_\epsilon(w_2)-\hat{\alpha}_\epsilon(w_1)}{w_2-w_1} \ge -1$ for all $\epsilon>0$ and all $w_2>w_1$.
  The continuity of $f$ implies that 
  $$\lim_{\epsilon \to 0} \hat{\alpha}_\epsilon(w) = \lim_{\epsilon\to 0} (1+\epsilon) f(s^{-1}_{1+\epsilon,1+\epsilon}(1,0,w)) = f(1,0,w)=\alpha(w),$$
  so
  $$\frac{\alpha(w_2)-\alpha(w_1)}{w_2-w_1} = \lim_{\epsilon \to 0} \frac{\hat{\alpha}_\epsilon(w_2)-\hat{\alpha}_\epsilon(w_1)}{w_2-w_1} \ge -1.$$
  This proves the lemma.
\end{proof}

\subsection{Minimality of graphical strips}\label{sec:minimality}

In this section, we show that monotone subsets of $\HH$ are perimeter-minimizing and give a criterion for $S_\alpha^+$ to be monotone. We first recall some definitions and formulas. Let $\cL$ be the space of horizontal lines in $\HH$, and for $U\subset \HH$, let $\cL(U)$ denote the set of horizontal lines that intersect $U$. Let $\cL$ be the measure on $\cN$ that is invariant under isometries of $\HH$ (rotations around the $z$--axis, left-translations, and maps $s_{\pm1,\pm1}$). This measure is unique up to constants, and we normalize it so that $\cN(\cL(B(\0,r)))=r^3$ for every $r>0$.

The kinematic formula (see~\cite{Mon05} or equation (6.1) in~\cite{CKN})  relates perimeter on lines to perimeter in $\HH$ as follows. There is a constant $c>0$ such that for any set $E\subset \HH$ with locally finite perimeter and any open subset $W\subset \HH$, $\Per_{E\cap L}(W\cap L)<\infty$ for $\cN$--almost every line $L$ and 
\begin{equation}\label{eq:kinematic}
  \Per_E(W)=c\int_{\cL} \Per_{E\cap L}(W\cap L) \ud\cN(L).
\end{equation}
Here, $\Per_{E\cap L}$ is the perimeter of $E\cap L$ as a subset of $L$. If $\Per_{E\cap L}$ is finite, then there is a finite union of intervals $S\subset L$ such that $(E\symdiff S)\cap L$ has measure zero and $\Per_{E\cap L}$ is the counting measure on $\partial S$.

We say that a subset $W\subset \HH$ is convex if it is convex as a subset of $\R^3$. For any $g\in \HH$, the map $h\mapsto gh$ is affine, so convexity is preserved by left-translation. Let $E\subset \HH$ be a set with locally finite perimeter and let $W\subset \HH$ be a convex open set. 
We say that $E$ is \emph{monotone} on $W$ if for almost every $L\in \cL$, we have $\Per_{E\cap L}(W\cap L)\le 1$. That is, up to a set of measure zero, $E\cap W\cap L$ is one of $\emptyset$, $W\cap L$, or the intersection of $W\cap L$ with a ray. Monotonicity is preserved by left-translation.

\begin{prop}\label{prop:monotone-implies-min}
  Let $E\subset \HH$ be a set with locally finite perimeter and let $W\subset \HH$ be a convex open set such that $E$ is monotone on $W$. Then $E$ is perimeter-minimizing in $W$.
\end{prop}
\begin{proof}
  Let $F\subset \HH$ be a set with locally finite perimeter such that $E\symdiff F\Subset W$. We claim that
  \begin{equation}\label{eq:more-int}
    \Per_{E\cap L}(W\cap L)\le \Per_{F\cap L}(W\cap L)
  \end{equation}
  for almost every line $L\in \cL$.

  For almost every line $L\in \cL$, we have $\Per_{E\cap L}(W\cap L)\le 1$. If $\Per_{E\cap L}(W\cap L)=0$, then \eqref{eq:more-int} holds, so we suppose that $\Per_{E\cap L}(W\cap L)=1$. Then $W\cap \partial_{\cH^1}(E\cap L)$ consists of a single point, say $p$, and there is a ray $R$ based at $p$ such that $E\cap W\cap L=R\cap W\cap L$, up to a null set.

  Since $W$ is convex, the intersection $W\cap L$ is an interval, with (up to null sets) one end of $W\cap L$ in $E$ and the other end outside of $E$. Since $E\symdiff F\Subset W$, we likewise have one end of $W\cap L$ in $F$ and the other end outside of $F$. Therefore, $\Per_{F\cap L}(W\cap L)\ge 1=\Per_{E\cap L}(W\cap L)$.

  By \eqref{eq:kinematic},
  $$\Per_E(W)=c\int_{\cL} \Per_{E\cap L}(W\cap L) \ud\cN(L) \le \Per_F(W),$$
  so $E$ is perimeter-minimizing on $W$.
\end{proof}

Finally, we apply this to $S_\alpha$.
\begin{lemma}\label{lem:bound-implies-min}
  Let $S_\alpha$ be a graphical strip over $K$ and suppose that $\frac{\alpha(t)-\alpha(s)}{t-s} \ge -1$ for all $s<t$. Then $S_\alpha^+$ is perimeter-minimizing on $U=\{-1<x<1\}\subset \HH$.
\end{lemma}
\begin{proof}
  We claim that $S_\alpha^+$ is monotone on $U$. By Lemma~\ref{lem:alternate}, 
  $$S_\alpha = \bigcup_{z\in \R} [(-1, -\sigma(z), z), (1, \sigma(z), z)]$$
  for some Lipschitz function $\sigma\from \R\to \R$ such that $-2 \le \sigma'(z) < 2$ for all $z$.

  For $x\in [-1,1]$, $z\in \R$, let $\Phi(x,z) = (x, x \sigma(z), z)$ parametrize $S_\alpha$. It suffices to show that any horizontal line that does not contain a ruling of $S_\alpha$ intersects $\Phi((-1,1)\times \R)$ at most once.

  Let $x_1,x_2\in (-1,1)$ and $z_1<z_2$. The horizontal plane $P_1$ centered at $\Phi(x_1,z_1)$ is the set
  \begin{align*}
    P_1 
    &= \left\{(x_1,x_1\sigma(z_1),z_1)\cdot (x-x_1,y-x_1\sigma(z_1),0) \mid x, y\in \R\right\} \\ 
    &= \left\{\left(x,y, z_1 + \frac{x_1 y - x x_1 \sigma(z_1)}{2}\right) \mid x, y\in \R\right\},
  \end{align*}
  so $\Phi(x_2,z_2)\in P_1$ if and only if 
  $$z_2 = z_1 + \frac{x_1 x_2\sigma(z_2) - x_2 x_1 \sigma(z_1)}{2},$$
  i.e.,
  $$x_1 x_2 \frac{\sigma(z_2)-\sigma(z_1)}{z_2-z_1} = 2.$$
  Since $\frac{|\sigma(z_2)-\sigma(z_1)|}{z_2-z_1} \le 2$ and $|x_1x_2|<1$, this is impossible. Therefore, if $L$ intersects $\Phi((-1,1)\times \R)$ at two points, then those two points have the same $z$--coordinate and thus lie on the same ruling of $S_\alpha$. Therefore, $S_\alpha^+$ is monotone on $U$ and thus perimeter-minimizing on $U$.
\end{proof}

Lemma~\ref{lem:min-implies-bound} and Lemma~\ref{lem:bound-implies-min} prove the two directions of Proposition~\ref{prop:strips}.

\section{Scaling limits of ruled graphs}\label{sec:scaling-limits}

In the last section, we classified area-minimizing graphical strips. In this section and the next section, we will use the classification of area-minimizing graphical strips to classify entire area-minimizing ruled intrinsic graphs. We first show that the scaling limit of an entire ruled intrinsic graph is a graphical strip; in fact, it is a broken plane (Definition~\ref{def:broken-plane}). 
\begin{lemma}\label{lem:scaling-limits-are-BPs}
  Let $f\from V_0 \to \R$ be a continuous function such that \(\Gamma_f\) is a ruled surface. There are $u\in [0,\infty]$ and $\theta\in \R$ such that if $\rot_\theta\from \HH\to \HH$ is rotation by $\theta$ around the $z$--axis, then $s_{t,t}^{-1}(\Gamma^+_f)\to \rot_\theta(\BP_u)$ as $t\to \infty$.
  Furthermore, if $u=0$, then $\Gamma_f$ is a vertical plane.
\end{lemma}
We write $E_i\to E$ if $E_i$ converges locally to $E$, that is, for every compact set $K\subset \HH$, we have $\cH^4(E_i\symdiff E)\to 0$. 

We prove Lemma~\ref{lem:scaling-limits-are-BPs} by analyzing the characteristic curves of $\Gamma_f$. As noted in Section~\ref{sec:char}, every ruling of $\Gamma_f$ projects to a parabola in $V_0$, and since $\Gamma_f$ is entire, these parabolas cover $V_0$.

For each ruling $R$, let $m(R)$ be the slope of $R$, let $w(R)=(0,w_y(R),w_z(R)) \in W_0$ be the point where $R$ intersects the $yz$--plane, and let $\lambda_R(t)=w(R)(X+m(R)Y)^t$ parametrize $R$. Let $g_R(t)=z(\Pi(\lambda_R(t)))$ so that $\Pi(R)$ is the graph $\{(x,0,g_R(x))\in V_0\mid x\in \R\}$; by \eqref{eq:point-slope},
\begin{equation}\label{eq:gR-form}
  g_R(x)=w_z(R) - w_y(R) x - \frac{m(R)}{2} x^2.
\end{equation}

We first prove some lemmas characterizing $\Gamma_f$.
\begin{lemma}\label{lem:monotone-slopes}
  Let $R_1$ and $R_2$ be distinct rulings of \(\Gamma_f\) and suppose that $w_z(R_1)< w_z(R_2)$. Then $g_{R_1}(x) \le g_{R_2}(x)$ for all $x\in \R$ and $m(R_1) \ge m(R_2)$.
\end{lemma}
\begin{proof}
  We have $g_{R_1}(0) = w_z(R_1) < w_z(R_2) = g_{R_2}(t)$, so by Lemma~\ref{lem:ordering-rulings}, $g_{R_1}(t)\le g_{R_2}(t)$ for all $t$. By \eqref{eq:gR-form}, this implies $m(R_1) \ge m(R_2)$.
\end{proof}

For $p\in \HH$, $m\in \R$, let $L_{p,m}$ be the horizontal line of slope $m$ through $p$.
\begin{lemma}\label{lem:ruled-char-pts}
  Let $R_1$ and $R_2$ be distinct rulings of $\Gamma_f$ that intersect at a point $p\in \HH$; suppose that $m(R_1) <  m(R_2)$. Then for every $m\in [m(R_1), m(R_2)]$, $L_{p,m}$ is a ruling of $\Gamma_f$.
\end{lemma}
\begin{proof}
  After a translation, we may suppose that $p=\0$, so that $R_i=L_{\0,m(R_i)}$ and $\Pi(R_i)$ is the graph of the function $g_i(x)=- \frac{m(R_i)}{2} x^2$. Let $m\in (m(R_1), m(R_2))$, $q=(1,0, - \frac{m}{2})$, and let $M$ be the ruling through $\Psi_f(q)$. 
  By Lemma~\ref{lem:ordering-rulings}, $g_{R_2}(x)\le g_M(x) \le g_{R_1}(x)$, and since the graphs of $g_{R_1}$ and $g_{R_1}$ are tangent to the $x$--axis at $0$, so is the graph of $g_M$. Since $g_M(1)=-\frac{m}{2}$, we have $g_M(x)=- \frac{m}{2} x^2$, so $M=L_{\0,m}$.
\end{proof}

Combining these two lemmas, we get the following characterization of entire ruled graphs.
\begin{cor}\label{cor:entire-ruled}
  Let $\sigma_+,\sigma_-\from \R\to \R$,
  $$\sigma_+(z) = \max \{m\mid L_{Z^z,m}\subset \Gamma_f\}$$
  $$\sigma_-(z) = \min \{m\mid L_{Z^z,m}\subset \Gamma_f\}.$$
  Since $\Gamma_f$ is closed, these minima and maxima exist. Then $\sigma_+$ and $\sigma_-$ are nonincreasing, and $L_{Z^z,m}\subset \Gamma_f$ if and only if $m\in [\sigma_-(z),\sigma_+(z)]$.
\end{cor}
% That is, 
% $$S=\{(z,m)\in \R^2 \mid L_{Z^z,m} \subset \Gamma_f\}$$
% can be written as the boundary of the epigraph of $\sigma_-(z)$, i.e.,
% $$S=\partial \{(z,m)\in \R^2\mid m<\sigma_-(z)\}.$$
% Since $\sigma_-(z)$ is nonincreasing, rotating its epigraph by $\frac{\pi}{4}$ results in the epigraph of a $1$--Lipschitz function. It follows that $S$ can be parametrized by a Lipschitz function $\gamma=(\gamma_z,\gamma_m) \from \R\to \R^2$ such that $\gamma_z$ is non-decreasing and $\gamma_m$ is non-increasing.

% search "Hausdorff dimension of graph of increasing function"

Let $m_\infty=\lim_{t\to \infty}\sigma_+(t)$ and $m_{-\infty}=\lim_{t\to -\infty}\sigma_+(t)$.
These limits determine the scaling limit of $\Gamma_f$.
\begin{lemma}\label{lem:limit-f}
  For $t>0$, let $f_t(p)=t^{-1}f(s_{t,t}(p))$ so that $\Gamma_{f_t}=s_{t,t}^{-1}(\Gamma_f)$. For $x\ne 0$, let
  $$F(x,0,z) = \begin{cases}
    m_{-\infty} x & \text{if }m_{-\infty} <\infty \text{ and } z \le - \frac{m_{-\infty}}{2} x^2 \\
    -\frac{2 z}{x}& \text{if } - \frac{m_{-\infty}}{2} x^2 < z < - \frac{m_{\infty}}{2} x^2 \\
    m_\infty x & \text{if }m_{\infty} > - \infty \text{ and } z \ge - \frac{m_{\infty}}{2} x^2.
  \end{cases}$$
  Then $f_t$ converges to $F$ pointwise almost everywhere on $V_0$.
\end{lemma}
\begin{proof}
  Let $x_0,z_0\in \R$.
  We first consider the case that $x_0\ne 0$ and $z_0 \in (- \frac{m_{-\infty}}{2} x_0^2, -\frac{m_{\infty}}{2} x_0^2)$.   Let $R_1$, $R_2$ be rulings of $\Gamma_f$ such that $- \frac{m(R_1)}{2} x_0^2 < z_0 < - \frac{m(R_2)}{2} x_0^2$. Then
  $$g_{R_1}(tx_0) = - \frac{m(R_1)}{2} (tx_0)^2 + O(t) < t^2z_0$$
  when $t$ is sufficiently large; likewise, $g_{R_2}(tx_0) > t^2z_0$ when $t$ is sufficiently large.  That is, there is some $t_0>0$ such that $s_{t,t}(x_0,0,z_0)=(tx_0,0, t^2z_0)$ is between $\Pi(R_1)$ and $\Pi(R_2)$ when $t>t_0$.
  
  For each $t>0$, let $S_t$ be a ruling of $\Gamma_f$ such that $(tx_0,0,t^2z_0)\in \Pi(S_t)$.  When $t>t_0$, we have $g_{R_1}(tx_0) < t^2z_0 < g_{R_2}(tx_0)$ and thus $g_{R_1} \le g_{S_t} \le g_{R_2}$ on all of $\R$. We will use this inequality to bound the coefficients of $g_{S_t}$.

  Let $C>0$ be such that $|g_{R_1}(x)|<C$ and $|g_{R_2}(x)|<C$ for all $x\in [-1,1]$. Then $|g_{S_t}(x)|<C$ for all $x\in  [-1,1]$. By \eqref{eq:gR-form}, this impies $|w_z(S_t)|=|g_{S_t}(0)| < C$ and
  $$|w_y(S_t)| = \left| \frac{g_{S_t}(-1)-g_{S_t}(1)}{2}\right| < C.$$
  By \eqref{eq:gR-form} with $x=t x_0$,
  $$g_{S_t}(tx_0) + \frac{m(S_t)}{2} (tx_0)^2 
  = w_z(S_t) - w_y(S_t) t x_0,$$
  so
  $$\left|t^2 z_0+\frac{m(S_t)}{2} (tx_0)^2\right| 
  \le C + C |tx_0|,$$
  and
  $$\left|\frac{m(S_t)}{2} + \frac{z_0}{x_0^2}\right| \le \frac{C}{|tx_0|^2} + \frac{C}{|tx_0|}.$$
  That is, $\lim_{t\to \infty} m(S_t) = - \frac{2z_0}{x_0^2}$.
  Since the graph of $g_{S_t}$ is a characteristic curve, $g_{S_t}$ satisfies the differential equation
  $g_{S_t}'(x) = -f(x,0,g_{S_t}(x))$. Therefore, 
  \begin{multline*}
    \lim_{t\to\infty} f_t(x_0,0,z_0)=\lim_{t\to\infty} t^{-1} f(tx_0,0,t^2z_0)=\lim_{t\to\infty}-t^{-1} g_{S_t}'(t x_0) \\
    = \lim_{t\to\infty} -t^{-1}\left(- m(S_t) t x_0 - w_y(S_t)\right)
    = \lim_{t\to\infty}  m(S_t) x_0 =-\frac{2z_0}{x_0},
  \end{multline*}
  as desired.

  Next, we consider the case that $m_\infty>-\infty$ and $\frac{z_0}{x_0^2}>-\frac{m_\infty}{2}$. Let $\epsilon\in (0,1)$ and let $R_1$ be a ruling of $\Gamma_f$ such that $m_\infty \le m(R_1) < m_\infty + \epsilon$. Then
  % $g_{R_1}(x)= - \frac{m(R_1)}{2} x^2 + O(x)$ and since $-\frac{m(R_1)}{2} < \frac{z_0}{x_0^2}$,
  there is a $t_0$ such that
  \begin{equation}\label{eq:gR1-x0}
    -\frac{m_\infty + \epsilon}{2} t^2x_0^2< g_{R_1}(tx_0)  < t^2 z_0
  \end{equation}
  when $t>t_0$.

  Let $t>t_0$ and let $S_t$ be a ruling of $\Gamma_f$ such that $(tx_0,0,t^2z_0)\in \Pi(S_t)$. 
  By Lemma~\ref{lem:ordering-rulings}, $g_{R_1}(x) \le g_{S_t}(x)$ for all $x\in \R$, and by Lemma~\ref{lem:monotone-slopes}, $m(R_1) \ge m(S_t) \ge m_\infty$. Let $h(x) = g_{S_t}(x) - g_{R_1}(x)$. We have $h(x)\ge 0$ for all $x$ and $h''(x)=m(R_1)-m(S_t) \in [0,\epsilon]$. Since $h$ is a quadratic, for any $w\in \R$,
  $$0 \le h(tx_0 + w) = h(tx_0) + h'(t x_0) w + \frac{1}{2} h''(t x_0)w^2 \le h(tx_0) + h'(t x_0) w + \frac{\epsilon}{2} w^2.$$
  Letting $w= - \epsilon^{-1} h'(t x_0)$, this implies
  $$h(tx_0) - \frac{h'(t x_0)^2}{2 \epsilon} \ge 0$$
  and thus $|h'(t x_0)|\le \sqrt{2\epsilon h(t x_0)} = \sqrt{2\epsilon (t^2z_0-g_{R_1}(t x_0))}$. By \eqref{eq:gR1-x0},
  $$|h'(t x_0)| \le \sqrt{2 \epsilon t^2 \left(z_0 +     \frac{m_\infty + \epsilon}{2} x_0^2\right)} \le C|t|\sqrt{\epsilon},$$
  where $C=\sqrt{2z_0 + (m_\infty + 1) x_0^2}$.

  Therefore,
  $$\left|f_t(x_0,0,z_0) - \frac{-g'_{R_1}(tx_0)}{t}\right| = \left|\frac{-g'_{S_t}(tx_0)}{t} - \frac{-g'_{R_1}(tx_0)}{t}\right| = |t|^{-1} |h'(tx_0)| \le C\sqrt{\epsilon}.$$
  When $t>\max\{t_0,\epsilon^{-1}|w_y(R_1)|\}$,
  \begin{align*}
    \big|f_t(x_0,0,z_0)
    & - m_\infty x_0\big| \\
    & \le \left|f_t(x_0,0,z_0) - \frac{-g'_{R_1}(tx_0)}{t}\right| + \left|\frac{-g'_{R_1}(tx_0)}{t} - m(R_1) x_0\right| + |m(R_1) - m_\infty||x_0| \\
    & \le C \sqrt{\epsilon} + \frac{|w_y(R_1)|}{t} + \epsilon |x_0| \\
    &\le C \sqrt{\epsilon} + \epsilon(1+|x_0|).
  \end{align*}
  Since $\epsilon$ was arbitrary, $\lim_{t\to \infty} f_t(x_0,0,z_0) = m_\infty x_0$, as desired.  The case that $z_0<-\frac{m_{-\infty}}{2} x_0^2$ is similar.
\end{proof}

Finally, we prove Lemma~\ref{lem:scaling-limits-are-BPs}.
\begin{proof}[Proof of Lemma~\ref{lem:scaling-limits-are-BPs}]
  For $-\infty \le M\le m \le \infty$, let
  $$F_{m,M}(x,0,z) = \begin{cases}
    m x & \text{if }m <\infty \text{ and } z \le - \frac{m}{2} x^2 \\
    -\frac{2 z}{x}& \text{if } - \frac{m}{2} x^2 < z < - \frac{M}{2} x^2 \\
    M x & \text{if }M > - \infty \text{ and } z \ge - \frac{M}{2} x^2.
  \end{cases}$$
  By Lemma~\ref{lem:limit-f}, $f_t$ converges pointwise almost everywhere to $F:=F_{m_{-\infty},m_\infty}$, so $\one_{\Gamma_{f_t}^+}$ converges weakly to $\one_{\Gamma_{F}^+}$. We claim that for any $m,M$ there are $\theta$ and $u$ such that, up to a set of measure zero, $\Gamma_{F_{m,M}}^+=\rot_\theta(\BP_u^+)$.

  Suppose that $m$ and $M$ are both finite. Then $\Gamma_F$ contains 
  the intrinsic graph of $a_{M}(x,0,z)=Mx$ over the set $\{z\le -M x^2\}$. This is the upper half-plane through $\0$ of slope $M$, bounded by the $xy$--plane. Likewise, $\Gamma_F$ contains the lower half-plane through $\0$ of slope $m$, bounded by the $xy$--plane. Finally, $\Gamma_F$ contains the graph of $h(x,0,z)=-\frac{2 z}{x}$ over the set $\{- \frac{m}{2} x^2 < z < - \frac{M}{2} x^2\}$; this is the union of the horizontal lines through the origin with slopes between $m$ and $M$. It follows that, up to a set of measure $0$,  $\Gamma_{F_{m,M}}^+$ consists of two quadrants of $\HH$, one above the $xy$--plane, bounded by the vertical plane through $\0$ of slope $M$, and one below the $xy$--plane, bounded by the vertical plane through $\0$ of slope $m$.

  Thus, if $m_{-\infty}$ and $m_\infty$ are finite, then $\Gamma_F^+$ is a union of two quadrants. If $m_{-\infty}$ or $m_\infty$ are infinite, then $F$ is a limit of $F_{m,M}$'s. As $m\to m_{-\infty}$ and $M\to m_\infty$, $\Gamma_{F_{m,M}}^+$ converges to a union of two quadrants with slopes $m_{-\infty}$ and $m_{\infty}$, so $\Gamma_F^+$ is again a union of two quadrants.  Any such union can be written as $\rot_\theta(\BP_u^+)$ for some $u\in [0,\infty]$ and $\theta\in \R$, so $\Gamma_{f_t}^+\to \rot_\theta(\BP_u^+)$ as desired.
  
  Finally, suppose that $u=0$, so that $\rot_\theta(\BP_u)$ is a vertical plane $V$. In this case, $m_\infty=m_{-\infty}$; let $m=m_\infty$. Then, by Corollary~\ref{cor:entire-ruled}, every ruling of $\Gamma_f$ has slope $m$. By \eqref{eq:point-slope}, if $L_1$ and $L_2$ are two horizontal lines with slope $m$ and $w_y(L_1)\ne w_y(L_2)$, then $\Pi(L_1)$ and $\Pi(L_2)$ intersect transversely, so $L_1$ and $L_2$ cannot both be rulings of the same intrinsic graph. It follows that any two rulings $R_1,R_2\subset \Gamma_f$ have the same projection to the $xy$--plane, i.e.,
  $$\pi(R_1)=\pi(R_2)=\left\{(x,y,0)\in \HH \mid y=mx+ w_y(R_1)\right\}.$$
  Therefore, $\Gamma_f=\{(x,y,z)\in \HH\mid y=mx+ w_y(R_1)\}$ is a vertical plane.
\end{proof}

\section{Proof of Theorem~\ref{thm:ruled-bernstein}}\label{sec:proof-ruled-bernstein}

Now we prove Theorem~\ref{thm:ruled-bernstein}. We will need the following closure result for perimeter-minimizing subsets of $\HH$. 
This statement and proof is based on Theorem~21.14 of \cite{Maggi}.

\begin{prop}\label{prop:limits-of-min}
  Let $A\subset \HH$ be an open set and let
  $E_1,E_2,\dots\subset \HH$ be a sequence of perimeter minimizers in $A$ such that $E_i\cap A \to E\cap A$, where $\Per_E(U)<\infty$ for every $U\Subset A$. Then $E$ is a perimeter minimizer in $A$.
\end{prop}

We will need the following metric characterization of perimeter.
For a measurable set $E\subset \HH$, let $E^{(1)}$ be the set of density points of $E$, i.e,
$$E^{(1)}:=\left\{p\in \HH \mid \lim_{r\to 0^+} \frac{\cH^4(B(p,r)\cap E)}{\cH^4(B(p,r))} = 1\right\},$$
and let $E^{(0)}=(\HH\setminus E)^{(1)}$. 
Let $\partial_{\cH^4}E$ be the measure-theoretic boundary of $E$, i.e. 
$$\partial_{\cH^4}E:=\HH\setminus \left(E^{(1)} \cup E^{(0)}\right).$$
Franchi, Serapioni, and Serra Cassano \cite{FSSCRectifiability} showed that there is a left-invariant homogeneous metric $d_\infty$ and a corresponding spherical Hausdorff measure $\cS^3_\infty$ such that for any set $E\subset \HH$ with locally finite perimeter,
$$\Per_E(U) = \cS_\infty^3(U\cap \partial_{\cH^4}E).$$

\begin{proof}[Proof of Proposition~\ref{prop:limits-of-min}]
  Let $F\subset \HH$ be a set such that $E\symdiff F\Subset A$ and $\Per_F(U)<\infty$ for every $U\Subset A$. We claim that $\Per_E(A)\le \Per_F(A)$.

  Our first goal is to find a set $G$ such that $E\symdiff F \Subset G \Subset A$ and such that 
  \begin{equation}\label{eq:good-G1}
    \cS_\infty^3\left(\partial_{\cH^4} F\cap \partial G\right)=0
  \end{equation}
  and 
  \begin{equation}\label{eq:good-G2}
    \liminf_{i\to \infty} \cS_\infty^3\left(\big(E^{(1)}\symdiff E_i^{(1)}\big) \cap \partial G\right)=0.
  \end{equation}

  By the compactness of $E\symdiff F$, there are finitely many points $p_j$ and radii $r_j$ such that $E\symdiff F \Subset \bigcup_j B(p_j,r_j)$ and $B(p_j,2r_j)\Subset A$ for all $j$.
  
  We claim that there is a $t\in [1,2]$ such that $G_t=\bigcup_j B(p_j, tr_j)$ satisfies \eqref{eq:good-G1} and \eqref{eq:good-G2}. Since
  $\cS_\infty^3(\partial_{\cH^4} F\cap G_2) <\infty$, $G_t$ satisfies \eqref{eq:good-G1} for all but countably many $t\in [1,2]$.
  Furthermore, the coarea formula implies that
  $$\int_1^2 \cS_\infty^3\left(\big(E^{(1)}\symdiff E_i^{(1)}\big) \cap \partial B(p_j, t r_j)\right)\ud t \lesssim r_j^{-1} \cH^4\left(\big(E^{(1)}\symdiff E_i^{(1)}\big) \cap B(p_j,2r_j)\right).$$
  By Fatou's lemma, 
  \begin{align*}
    \int_1^2 &\liminf_{i\to \infty}  \cS_\infty^3\left(\big(E^{(1)}\symdiff E_i^{(1)}\big) \cap \partial G_t \right)\ud t \\
    & \le 
    \liminf_{i\to \infty} \int_1^2 \sum _j \cS_\infty^3\left(\big(E^{(1)}\symdiff E_i^{(1)}\big) \cap \partial B(p_j, t r_j) \right)\ud t \\
    &\lesssim \liminf_{i\to \infty}\,
    \sum_j r_j^{-1} \cH^4\left(\big(E^{(1)}\symdiff E_i^{(1)}\big) \cap B(p_j,2r_j)\right)
    = 0,
  \end{align*}
  so $G_t$ satisfies \eqref{eq:good-G2} for almost every $t\in [1,2]$. Let $G=G_t$ for some $t$ such that \eqref{eq:good-G1} and \eqref{eq:good-G2} are satisfied.

  Let $\overline{G}$ be the closure of $G$, and for each $i$, let 
  $F_i=(F\cap G)\cup (E_i\setminus G).$
%  $$F_i=\left(F\cap \overline{G}\right)\cup \left(E_i\setminus \overline{G}\right).$$
  To bound $\partial_{\cH^4} F_i$, note that $\partial_{\cH^4} F_i\cap G=\partial_{\cH^4} F\cap G$ and $\partial_{\cH^4} F_i\setminus \overline{G}=\partial_{\cH^4} E_i\setminus \overline{G}$. If $p\in \partial_{\cH^4} F_i\cap \partial G$, then either $p\in \partial_{\cH^4} F$, $p\in \partial_{\cH^4} E_i$, or $p\in E_i^{(1)}\symdiff F^{(1)}$. Therefore,
  \begin{equation}\label{eq:Fi-bdry}
    \partial_{\cH^4} F_i\subset \big(\partial_{\cH^4} F\cap \overline{G}\big) \cup \left(\partial_{\cH^4} E_i \setminus G\right) \cup \left(\big(E_i^{(1)}\symdiff F^{(1)}\big)\cap \partial G\right).
  \end{equation}
  
  Let $U$ be an open set such that $G\Subset U\Subset A$. Since $E_i$ is a perimeter minimizer and $E_i \symdiff F_i\Subset U$, 
  $$\Per_{E_i}(U) \le \Per_{F_i}(U) = \cS^3_\infty(U\cap \partial_{\cH^4}F_i).$$
  By \eqref{eq:Fi-bdry},
  $$\Per_{E_i}(U)\le \Per_{F}(\overline{G}) + \Per_{E_i}(U\setminus G) + \cS_\infty^3\left(\big(E^{(1)}\symdiff E_i^{(1)}\big)\cap \partial G\right).$$
  Subtracting $\Per_{E_i}(U\setminus G)$ from both sides,
  $$\Per_{E_i}(G) \le \Per_{F}(\overline{G}) + \cS_\infty^3\left(\big(E^{(1)}\symdiff E_i^{(1)}\big)\cap \partial G\right).$$
  
  By the lower semicontinuity of perimeter and equations \eqref{eq:good-G1}--\eqref{eq:good-G2},
  $$\Per_E(G)
    \le \liminf_{i\to \infty}\, \Per_{E_i}(G)
    % \le \Per_{F}(\overline{G}) + \liminf_{i\to \infty} \cS_\infty^3((E^{(1)}\symdiff E_i^{(1)})\cap \partial G)
    \le \Per_{F}(\overline{G})=\Per_{F}(G),$$
  as desired.
\end{proof}

Finally, we prove Theorem~\ref{thm:ruled-bernstein}.
\begin{proof}[Proof of Theorem~\ref{thm:ruled-bernstein}]
  Let $\Gamma$ be an entire area-minimizing ruled intrinsic graph and let $\Gamma^+$ be its epigraph, so that $\Gamma^+$ is perimeter-minimizing. By Lemma~\ref{lem:scaling-limits-are-BPs}, there are $\theta$ and $u\in [0,\infty]$ such that $s_{t,t}^{-1}(\Gamma^+)\to \rot_\theta(\BP_u)$ as $t\to \infty$. By Proposition~\ref{prop:limits-of-min}, $\rot_\theta(\BP^+_u)$ is perimeter-minimizing; we claim that $u=0$.

  By Theorem~\ref{thm:minimal-strips}, $\BP^+_u$ is not perimeter-minimizing when $u\in (0,\infty)$. We claim that $\BP^+_\infty$ is not perimeter-minimizing. Recall that $\BP_\infty^+=\{(x,y,z)\in \HH\mid xz > 0\}$; one calculates that
  \begin{align*}
    Y^{-t}\BP_\infty^+
    &=\left\{\left(x,y-t,z+\frac{tx}{2}\right)\in \HH\mid x z > 0\right\}\\
    &=\left\{(x,y,z)\in \HH\mid x \left(z-\frac{tx}{2}\right) > 0\right\}.
  \end{align*}
  That is, $\partial(Y^{-t}\BP_\infty^+)$ consists of the plane $\{x=0\}$ and the plane $\{z=\frac{tx}{2}\}$. As $t\to \infty$, these planes grow closer together. In fact, $\cH^4(\BP_\infty^+\cap B(Y^t,1)) \to 0$, $\cH^3(\BP_\infty^+\cap \partial B(Y^t,1))\to 0$, and $\Per_{\BP_\infty^+ }(B(Y^t,1))\gtrsim 1$. Let $C_t=\BP_\infty^+ \setminus B(Y^t,1)$; then
  $$\Per_{C_t}(B(Y^t,2)) = \Per_{\BP_\infty^+}(B(Y^t,2)\setminus B(Y^t,1)) + \cH^3(\BP_\infty^+\cap \partial B(Y^t,1)),$$
  and $\Per_{C_t}(B(Y^t,2)) < \Per_{\BP_\infty^+}(B(Y^t,2))$
  when $t$ is large. Thus $\BP^+_\infty$ is not area-minimizing.

  The only remaining possibility is that $s_{t,t}(\Gamma^+)\to \rot_\theta(\BP^+_0)$. By the second part of Lemma~\ref{lem:scaling-limits-are-BPs}, this implies that $\Gamma$ is a vertical plane.
\end{proof}

\section{Non-uniqueness of area-minimizers}\label{sec:non-unique}

In this section, we consider the boundary case of Theorem~\ref{thm:minimal-strips}, where $\sigma(z)=-2z$. By Theorem~\ref{thm:minimal-strips}, the surface
$$\Sigma = \bigcup_{z\in \R} [(-1, 2z, z), (1, -2z, z)]$$
is area-minimizing, but Theorem~\ref{thm:non-unique} asserts that there are uncountably many area-minimizing surfaces with the same boundary. In this section, we will prove Theorem~\ref{thm:non-unique}.

First, note that for all $z_1,z_2\in \R$, 
\begin{equation}\label{eq:all-rulings}
  (-1, 2z_1,z_1)^{-1} (1,-2z_2,z_2) = (1, - 2z_1, -z_1) (1,-2z_2,z_2) = (2, -2z_1-2z_2, 0),
\end{equation}
so any two points $(-1, 2z_1,z_1)$ and $(1, -2z_2,z_2)$ are connected by a horizontal line. For any surjective continuous increasing function $\rho\from \R\to \R$, let $\Sigma_\rho$ be the ruled surface
$$\Sigma_\rho = \bigcup_{z\in \R} [(-1, 2z,z), (1,-2\rho(z),\rho(z))].$$
Let
$$F_\rho(x,z) = (-1, 2z, z) \cdot (1,-z-\rho(z),0)^{x+1}$$
parametrize $\Sigma_\rho$ and let 
$U=\{(x,y,z)\in \HH\mid -1<x<1\}$. We claim that $\Sigma_\rho$ is an intrinsic graph whose epigraph is monotone in $U$.

\begin{lemma}\label{lem:non-unique-monotone}
  Let $a_1,a_2,b_1,b_2$ be such that $b_1<b_2$ and $a_1< a_2$. For $i=1,2$, let $R_i=((-1, 2a_i,a_i),(1, -2 b_i,b_i))$. There are no horizontal lines from $R_1$ to $R_2$.
\end{lemma}
\begin{proof}
  Our argument is based on the hyperboloid lemma proved in \cite[Lemma~2.4]{CheegerKleinerMetricDiff}, but we give a self-contained proof.
  
  Let $\pi\from \HH\to \R^2$, $\pi(x,y,z)=(x,y)$ be the projection to the $xy$--plane. Then $\pi(R_i)=((-1,2a_i), (1,-2b_i))$, and since $b_1<b_2$ and $a_1< a_2$, the projections $\pi(R_1)$ and $\pi(R_2)$ intersect at some point $p=(x_0,y_0)$. Let $q_i = (x_0,y_0,z_i) \in R_i$ be the point such that $\pi(q_i)=p$.

  Let $V_i=(1, -a_i-b_i, 0)$, $s_0=-1-x_0$ and $t_0=1-x_0$ so that
  $R_i=\{q_i V_i^r\mid r\in (s_0,t_0)\}.$
  Suppose that $q_1 V_1^s$ is connected to $q_2 V_2^t$ by a horizontal segment. Recalling that $V_1^s$ is just alternate notation for $sV_1$, we have $q_1 \cdot sV_1 \cdot W = q_2 \cdot t V_2$ for some $W\in \mathsf{A}$. Projecting to $\R^2$, we see that 
  $$\pi(q_1 \cdot sV_1 \cdot W) = p + sV_1 + W = p + t V_2,$$
  i.e., $W=tV_2 - s V_1$ and 
  $$sV_1\cdot (tV_2 - s V_1) \cdot (-t V_2) = (-q_1) \cdot q_2 = (z_2-z_1) Z.$$
  By the Baker--Campbell--Hausdorff formula, $V \cdot W = V + W + \frac{1}{2}[V,W]$, so
  $$sV_1\cdot (tV_2 - s V_1) \cdot (-t V_2) 
  = \frac{1}{2} [sV_1, tV_2 - s V_1] + \frac{1}{2} [t V_2, -tV_2] = \frac{st}{2} [V_1, V_2].$$
  We calculate that $[V_1,V_2]=wZ$, where $w = \frac{1}{2}(a_1-a_2+ b_1 - b_2)>0$, so $z_2-z_1=w s t$.

  By \eqref{eq:all-rulings}, there is a horizontal line from $q_1V_1^{s_0}=(-1,2a_1,a_1)$ to $q_2V_2^{t_0}=(1,-2b_2,b_2)$, so $z_2-z_1= w s_0 t_0$. Therefore, $p_1 V_1^{s}$ is connected to $p_2 V_2^{t}$ by a horizontal line if and only if $st=s_0t_0$. But there are no $s, t\in (s_0,t_0)$ such that $st=s_0t_0$, so there are no horizontal lines connecting $R_1$ and $R_2$.
\end{proof}

This implies that $\Sigma_\rho$ is an intrinsic graph, since any coset of $\langle Y\rangle$ intersects $\Sigma_\rho$ at most once. We claim that $\Sigma_\rho$ is an intrinsic graph over $K = [-1,1]\times \{0\} \times \R\subset V_0$. 
By \eqref{eq:point-slope},
$$\Pi([(- 1, 2z_1, z_1) (1,-2z_2,z_2)]) = \{(x,0,g_{z_1,z_2}(x))\mid x\in \R\}$$
where
\begin{multline}\label{eq:gz1z2}
  g_{z_1,z_2}(x) = 2 z_1 - 2z_1 (x+1) + \frac{z_1+z_2}{2} (x+1)^2 = \frac{z_1}{2}\left((x+1)^2 - 4 (x+1) + 4\right) + \frac{z_2}{2}(x+1)^2 \\
  = \frac{z_1}{2}(x-1)^2 + \frac{z_2}{2}(x+1)^2.
\end{multline}
Therefore, $\Pi(F_\rho(x,z)) = (x, 0 , g_{z,\rho(z)}(x))$. Since $\rho(z)$ is surjective, $z\mapsto g_{z,\rho(z)}(x)$ is surjective for any $x\in [-1,1]$. We conclude that $\Pi(\Sigma_\rho)=K$.

By Lemma~\ref{lem:non-unique-monotone}, the epigraph $\Sigma_\rho^+$ is monotone in $U$. Proposition~\ref{prop:monotone-implies-min} implies that $\Sigma_\rho$ is area-minimizing in $U$. This proves Theorem~\ref{thm:non-unique}. 

Finally, we calculate the area of $\Sigma_\rho$; since $\Sigma_\rho$ is area-minimizing in $U$, the area of $F_\rho([-1,1]\times[a,b])$ should depend only on $a$, $b$, $\rho(a)$ and $\rho(b)$.
Suppose that $\rho$ is Lipschitz.
By the area formula \eqref{eq:area-formula}, for any $a<b$,
\begin{multline*}
  \area F_\rho([-1,1]\times[a,b]) = \int_a^b \int_{-1}^1 \sqrt{1+(z+\rho(z))^2} \cdot \big|J[\Pi\circ F_\rho](x,z)\big|\ud x \ud z\\
  =\int_a^b \sqrt{1+(z+\rho(z))^2} j(z)\ud z,
\end{multline*}
where $j(z)=\int_{-1}^1 \big|J[\Pi\circ F_\rho](x,z)\big|\ud x$. By \eqref{eq:gz1z2}, we have
$$j(z) = \int_{-1}^1 \partial_z[g_{z,\rho(z)}(x)] \ud x 
= \int_{-1}^1 \frac{1}{2}(x-1)^2 + \frac{\rho'(z)}{2}(x+1)^2\ud x = \frac{4}{3}(1+\rho'(z)),$$
so, by substitution,
$$\area F_\rho([-1,1]\times[a,b]) = \frac{4}{3} \int_a^b \sqrt{1+(z+\rho(z))^2}(1+\rho'(z)) \ud z = \frac{4}{3} \int_{a+\rho(a)}^{b+\rho(b)} \sqrt{1+m^2} \ud m.$$

\section{Conjectural minimizers}\label{sec:competitors}

In this section, we construct the surfaces shown in Figure~\ref{fig:conj-surfs}, which we conjecture to be area-minimizing or energy-minimizing competitors for $\BP_u$.
We first construct the family of conjectural energy minimizers.
Given an intrinsic Lipschitz graph $\Gamma_f$ such that $f$ is defined on an open subset $U\subset V_0$, we define the \emph{intrinsic Dirichlet energy} of $f$ on $U$ by $E_U(f)=\frac{1}{2} \int_U (\nabla_ff)^2\ud \mu$, where $\mu$ is Lebesgue measure on $V_0$. We say that $f$ is \emph{energy-minimizing} on $U$ if, for any $r>0$, we have $E_{U\cap B(\0,r)}(f)\le E_{U\cap B(\0,r)}(g)$ for any intrinsic Lipschitz function $g$ such that $g-f\in C^0_c(U\cap B(\0,r))$. 

An intrinsic Lipschitz graph is \emph{harmonic} if it is a critical point of the energy with respect to contact deformations, which are smooth deformations of $\HH$ that preserve the horizontal distribution. These graphs were studied in \cite{YoungHGraphs}. Harmonic graphs can often be written as unions of horizontal lines which meet along horizontal curves. These graphs satisfy a slope condition; if two horizontal segments intersect on a horizontal curve, then the slope of the curve is the average of the slopes of the horizontal segments.

Let $W_u:=\BP_u\cap \{(x,y,z)\mid |x|\le 1\}$ and let $\partial_\pm W_u=\BP_u\cap \{x=\pm1\}$ be the two curves bounding $W_u$. We will construct a harmonic intrinsic graph $\Sigma_u^h$ which is bounded by $\partial W_u$ and ruled away from two singular curves; we conjecture that $\Sigma_u^h$ is an energy-minimizing filling of $\partial W_u$. 

The idea behind the construction of $\Sigma_u^h$ is that there is a horizontal segment $\lambda=[(-1,-u,-\frac{u}{2}), (1,-u,\frac{u}{2})]$ connecting $\partial_-W_u$ to $\partial_+W_u$. This separates $\partial W_u$ into two parts: one consisting of $\lambda$ and the upper portions of $\partial_\pm W_u$ and one consisting of $\lambda$ and the lower portions of $\partial_\pm W_u$. We will fill each part by a harmonic intrinsic graph.

\begin{figure}
  \def\svgwidth{.5\textwidth}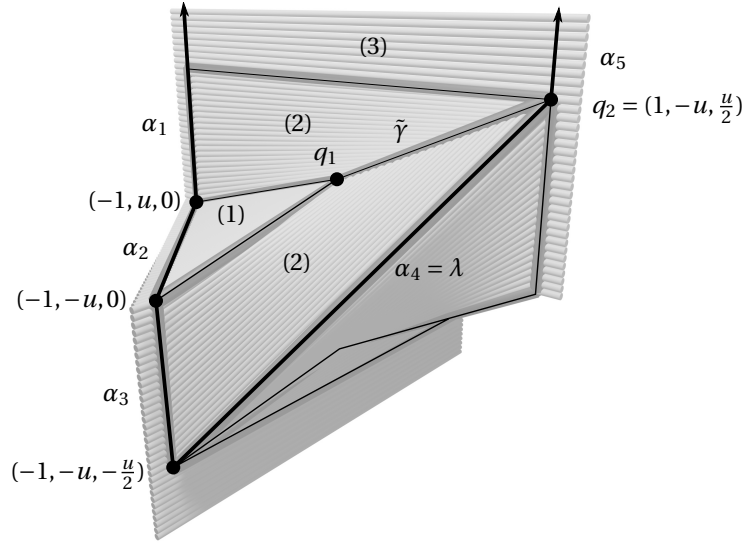
  \caption{\label{fig:energy-schematic} A schematic of $\Sigma^h_u$ with important features labeled.}
\end{figure}

Let $\alpha=\alpha(u)$ be the curve marked by a thick line in Figure~\ref{fig:energy-schematic}. This is the curve made up of five segments $\alpha_1,\dots, \alpha_5$,
\newcommand{\seg}[1]{\stackrel{#1}{\text{---}}}
\begin{multline*}
\alpha
=(-1,u,\infty) \seg{\alpha_1} (-1,u,0) \seg{\alpha_2} (-1,-u,0) \seg{\alpha_3} \left(-1,-u, -\frac{u}{2}\right) \seg{\alpha_4} \left(1,-u,\frac{u}{2}\right) \seg{\alpha_5} (1,-u,\infty).
\end{multline*}
(That is, $\alpha=\alpha_1\cup\dots\cup \alpha_5$, where $\alpha_1$ is the vertical ray pointing upward from $(-1,u,0)$, $\alpha_2=[(-1,u,0), (-1,-u,0)]$, and so on.)
Here, $\alpha_1$, $\alpha_3$, and $\alpha_5$ are vertical, $\alpha_4$ is horizontal, and $\alpha_2$ is neither.

The formula for the graph filling $\alpha$ is complicated, so we describe it in terms of its horizontal foliation.
Let $q_1=(0,-\frac{u}{2},\frac{u}{4})$. For any $y\in [-u,u]$, we have $q_1\cdot (-1,y+\frac{u}{2},0) = (-1,y,0)$, so there is a horizontal line from $q_1$ to any point in $\alpha_2$. Let $q_2=(1,-u,\frac{u}{2})$; then $q_1\cdot (1, -\frac{u}{2},0)=q_2$, so $q_1$ and $q_2$ are connected by a horizontal segment. The segment $[q_1,q_2]$ will be the characteristic set of $\Sigma^h_u$.

We write $\Sigma^h_u$ as a union of three families of horizontal segments, labeled (1), (2), and (3) in Figure~\ref{fig:energy-schematic}.
\begin{enumerate}
\item
  Segments connecting $q_1$ to the points of $\alpha_2$.
\item For each point $p\in [q_1,q_2]$, segments connecting $p$ to a point on $\alpha_1$ and a point on $\alpha_3$.
\item Segments with slope $-u$ connecting $(-1,u,z)$ to $(1,-u,z)$ for $z\in [\frac{u}{2},\infty)$. 
\end{enumerate}
For each $t\in [0,1]$, the second family contains two horizontal segments connecting $q_1\cdot (1, -\frac{u}{2},0)^t$ to $\alpha_1$ and $\alpha_3$. The projections of these segments to $\mathsf{A}$ connect $(t,-\frac{u}{2}(t+1))$ to $(-1,\pm u)$, so they have slope $-\frac{u}{2}\pm \frac{u}{t+1}$, while the characteristic nexus of $\Sigma^h_u$ has slope $-\frac{u}{2}$. That is, $\Sigma^h_u$ satisfies the slope condition in \cite{YoungHGraphs}. By Remark~3.10 of \cite{YoungHGraphs}, $\Sigma_u$ is an intrinsic Lipschitz graph, so $\Sigma^h_u$ is a harmonic graph. The full surface in Figure~\ref{fig:conj-surfs} is a union $\Sigma^h_u\cup s_{-1,1}(\Sigma^h_u)$ of two copies of $\Sigma^h_u$, and we conjecture that this is the minimal-energy surface filling $\partial W_u$.

Note that for any $u>0$, we can write $\Sigma^h_u=s_{1,u}(\Sigma^h_1)$. Stretch automorphisms send energy minimizers to energy minimizers \cite[3.2]{YoungHGraphs}, so proving that $\Sigma^h_1$ minimizes energy would imply that all of the $\Sigma^h_u$'s minimize energy.

Now we construct conjectural area-minimizing surfaces $\Sigma_u$ filling $\partial W_u$. These look similar to the harmonic surfaces constructed above; the main differences are that the characteristic set is the horizontal lift of a hyperbola rather than a horizontal line segment and that $\Sigma_u$ contains a family of horizontal segments of slope $0$, all parallel to $\alpha_4$ (labeled (3) in Figure~\ref{fig:schematic}).

\begin{figure}
  \def\svgwidth{.5\textwidth}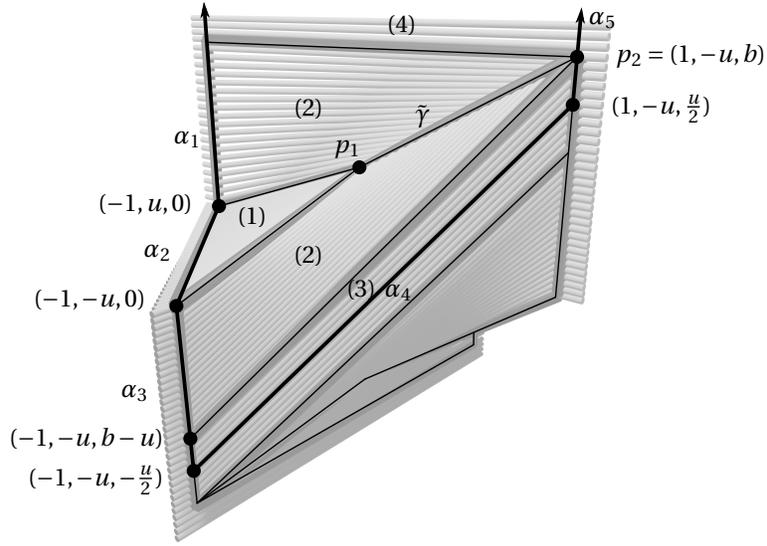
  \caption{\label{fig:schematic} A schematic of $\Sigma_u$, with important features labeled.}
\end{figure}

We construct $\Sigma_u$ out of area-minimizing $Z$--graphs and vertical rectangles.
As described in \cite{ChengHwangYang}, an area-minimizing $Z$--graph $\Sigma$ can be written as a union of horizontal line segments. These segments can meet along horizontal curves in $\Sigma$ as long as the tangent line to the horizontal curve bisects the angle between the horizontal line segments. Such singularities form the characteristic set of the surface.

\begin{figure}
  \begin{center}
    \def\svgwidth{.5\textwidth}
    %% Creator: Inkscape 1.0.2 (e86c8708, 2021-01-15), www.inkscape.org
%% PDF/EPS/PS + LaTeX output extension by Johan Engelen, 2010
%% Accompanies image file 'areaMinAOverlay.pdf' (pdf, eps, ps)
%%
%% To include the image in your LaTeX document, write
%%   \input{<filename>.pdf_tex}
%%  instead of
%%   \includegraphics{<filename>.pdf}
%% To scale the image, write
%%   \def\svgwidth{<desired width>}
%%   \input{<filename>.pdf_tex}
%%  instead of
%%   \includegraphics[width=<desired width>]{<filename>.pdf}
%%
%% Images with a different path to the parent latex file can
%% be accessed with the `import' package (which may need to be
%% installed) using
%%   \usepackage{import}
%% in the preamble, and then including the image with
%%   \import{<path to file>}{<filename>.pdf_tex}
%% Alternatively, one can specify
%%   \graphicspath{{<path to file>/}}
%% 
%% For more information, please see info/svg-inkscape on CTAN:
%%   http://tug.ctan.org/tex-archive/info/svg-inkscape
%%
\begingroup%
  \makeatletter%
  \providecommand\color[2][]{%
    \errmessage{(Inkscape) Color is used for the text in Inkscape, but the package 'color.sty' is not loaded}%
    \renewcommand\color[2][]{}%
  }%
  \providecommand\transparent[1]{%
    \errmessage{(Inkscape) Transparency is used (non-zero) for the text in Inkscape, but the package 'transparent.sty' is not loaded}%
    \renewcommand\transparent[1]{}%
  }%
  \providecommand\rotatebox[2]{#2}%
  \newcommand*\fsize{\dimexpr\f@size pt\relax}%
  \newcommand*\lineheight[1]{\fontsize{\fsize}{#1\fsize}\selectfont}%
  \ifx\svgwidth\undefined%
    \setlength{\unitlength}{400bp}%
    \ifx\svgscale\undefined%
      \relax%
    \else%
      \setlength{\unitlength}{\unitlength * \real{\svgscale}}%
    \fi%
  \else%
    \setlength{\unitlength}{\svgwidth}%
  \fi%
  \global\let\svgwidth\undefined%
  \global\let\svgscale\undefined%
  \makeatother%
  \begin{picture}(1,1)%
    \lineheight{1}%
    \setlength\tabcolsep{0pt}%
    \put(0,0){\includegraphics[width=\unitlength,page=1]{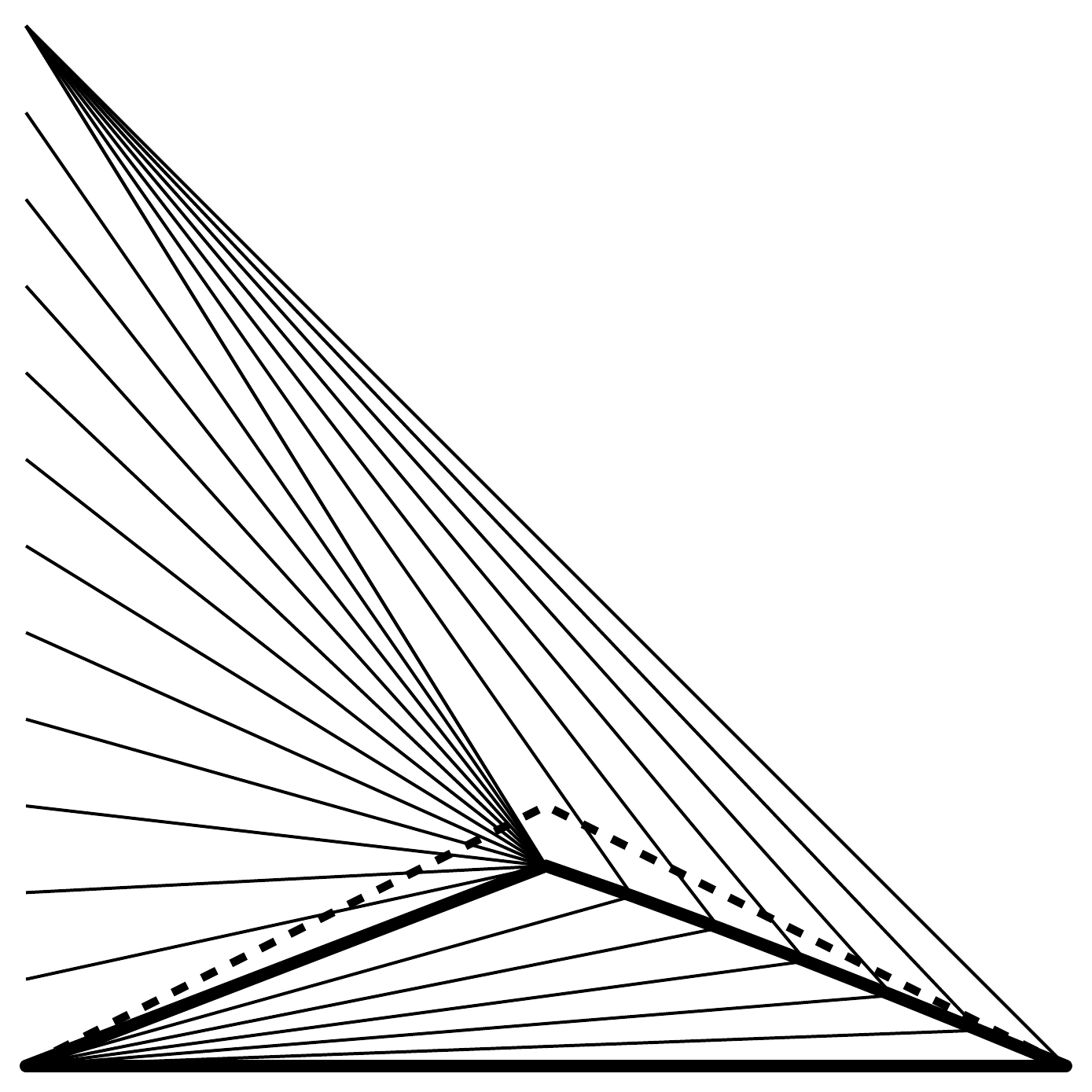}}%
    \put(-0.02053429,0.96369204){\color[rgb]{0,0,0}\makebox(0,0)[rt]{\lineheight{1.25}\smash{\begin{tabular}[t]{r}$(-1,u)$\end{tabular}}}}%
    \put(-0.01678429,0.01178127){\color[rgb]{0,0,0}\makebox(0,0)[rt]{\lineheight{1.25}\smash{\begin{tabular}[t]{r}$(-1,-u)$\end{tabular}}}}%
    \put(0,0){\includegraphics[width=\unitlength,page=2]{areaMinAOverlay.pdf}}%
    \put(0.59052023,0.11330196){\color[rgb]{0,0,0}\makebox(0,0)[lt]{\lineheight{1.25}\smash{\begin{tabular}[t]{l}$\gamma$\end{tabular}}}}%
    \put(1.01144657,0.01178127){\color[rgb]{0,0,0}\makebox(0,0)[lt]{\lineheight{1.25}\smash{\begin{tabular}[t]{l}$(1,-u)$\end{tabular}}}}%
    \put(0,0){\includegraphics[width=\unitlength,page=3]{areaMinAOverlay.pdf}}%
    \put(-0.02053429,0.49638033){\color[rgb]{0,0,0}\makebox(0,0)[rt]{\lineheight{1.25}\smash{\begin{tabular}[t]{r}$(-1,0)$\end{tabular}}}}%
  \end{picture}%
\endgroup%

  \end{center}
  \caption{\label{fig:areaMinDiagram} The projections to $\mathsf{A}$ of the horizontal segments that make up $\Sigma_u$.}
\end{figure}

Figure~\ref{fig:areaMinDiagram} shows the projection to the $xy$--plane $\mathsf{A}$ of the horizontal segments that make up $\Sigma_u$.
The characteristic set of $\Sigma_u$, marked by one of the thick lines in Figure~\ref{fig:areaMinDiagram},
is a lift of a segment of the hyperbola in $\mathsf{A}$ with foci at $(-1,\pm u)$ that passes through the point $(1,-u)$. We call this hyperbola $\gamma$; one can calculate that $\gamma$ satisfies the equation
$$\frac{y^2}{(\sqrt{u^2+1}-1)^2} - \frac{(x+1)^2}{u^2-(\sqrt{u^2+1}-1)^2}=1.$$
Let
$$a=-\left(\sqrt{u^2+1}-1\right)\sqrt{\frac{1}{u^2-(\sqrt{u^2+1}-1)^2}+1}$$
so that $(0,a)\in \gamma$. Let $L=[(-1,0), (1,-u)]$ be the long dashed line in Figure~\ref{fig:areaMinDiagram}. The segment of $\gamma$ from $(0,a)$ to $(1,-u)$ lies below $L$, so $a<-\frac{u}{2}$.

We construct $\Sigma_u$ by lifting the lines in Figure~\ref{fig:areaMinDiagram} to $\HH$. Let $p_1=(0,a,\frac{a}{2})$.
Then for any $y\in [-u,u]$, we have $p_1\cdot (-1,y-a,0) = (-1,y,0)$, so the lines on the left of Figure~\ref{fig:areaMinDiagram} lift to horizontal lines from the points of $\alpha_2$ to $p_1$.
Let $\tilde{\gamma}$ be the lift of $\gamma$ that passes through $p_1=(0,a,\frac{a}{2})$, and let $b$ be such that $p_2=(1,-u,b)\in \tilde{\gamma}$. The concatenation of $[(-1,-u,0),p_1]$, $\tilde{\gamma}$, and $[p_2, p_2 X^{-2}]$ is a horizontal curve in $\Sigma_u$ that connects $(-1,-u,0)$ to $p_2 X^{-2} = (-1,-u,b-u)$ and projects to the thick curve in Figure~\ref{fig:areaMinDiagram}. The area of this curve is equal to $u-b$, and the curve is contained in the triangle with vertices $(-1,-u), (1,-u)$, and $(0,-\frac{u}{2})$ (marked by dashed lines in Figure~\ref{fig:areaMinDiagram}), so $0 < u - b < \frac{u}{2}$ and $\frac{u}{2}<b<u$. That is, the endpoint of $\tilde{\gamma}$ lies in the interior of $\alpha_3$, as illustrated in Figure~\ref{fig:schematic}.

The horizontal segments in $\Sigma_u$ can be divided into four families, labeled (1), (2), (3), and (4) in Figure~\ref{fig:schematic}:
\begin{enumerate}
\item Segments connecting $p_1$ to the points of $\alpha_2$.
\item For each point $p\in \tilde{\gamma}$, segments connecting $p$ to a point on $\alpha_1$ and a point on $\alpha_3$. The projections of these segments to $\mathsf{A}$ connect points on the hyperbola $\gamma$ to its two foci.  
\item Segments with slope $0$ connecting $(-1,-u,z-u)$ to $(1,-u,z)$ for $z\in [\frac{u}{2},b]$. 
\item Segments with slope $-u$ connecting $(-1,u,z)$ to $(1,-u,z)$ for $z\in [b,\infty)$. 
\end{enumerate}
Since the tangent line to a hyperbola at a point bisects the lines from the point to the foci of the hyperbola, this surface satisfies the angle condition in Section~7 of \cite{ChengHwangYang} and is thus an area-minimizer.
The full surface in Figure~\ref{fig:conj-surfs} is a union $\Sigma_u\cup s_{-1,1}(\Sigma_u)$ of two copies of $\Sigma_u$, and we conjecture that this is the area-minimizing surface filling $\partial W_u$.

\bibliographystyle{alphaurl}
\bibliography{bernstein}

\end{document}